\documentclass[11pt, reqno]{amsart}
\usepackage{amssymb,latexsym,amsmath,amsfonts,mathdots,enumitem}
\usepackage{latexsym}
\usepackage[mathscr]{eucal}
\usepackage{colortbl,xcolor}
\usepackage{lmodern}
\usepackage{sansmathaccent}
\usepackage[latin1]{inputenc}
\usepackage{tikz}
\usepackage{physics}
\usetikzlibrary{shapes,arrows}
\allowdisplaybreaks
\usetikzlibrary{matrix,calc,shapes,arrows,positioning}
\pdfmapfile{+sansmathaccent.map}

\numberwithin{equation}{section}
\theoremstyle{plain}

\newtheorem{thm}{Theorem}[section]

\newtheorem{lem}[thm]{Lemma}
\newtheorem{prop}[thm]{Proposition}
\theoremstyle{definition}
\newtheorem{defn}[thm]{Definition}
\newtheorem{rem}[thm]{Remark}

\numberwithin{equation}{section}

\def\beq{\begin{eqnarray}}
	\def\eeq{\end{eqnarray}}
\def\beqa{\begin{eqnarray*}}
	\def\eeqa{\end{eqnarray*}}

\def\beqn{\begin{equation}}
	\def\eeqn{\end{equation}}

\def\mg#1{}

\renewcommand{\epsilon}{\varepsilon}
\renewcommand{\phi}{\varphi}

\renewcommand{\bf}[1]{\textbf{#1}}
\renewcommand{\it}[1]{\textit{#1}}
\renewcommand{\sc}[1]{\textsc{#1}}
\renewcommand{\sf}[1]{\textsf{#1}}

\numberwithin{equation}{section}
\allowdisplaybreaks[4] 

\setlist[enumerate]{font=\upshape,noitemsep, topsep=0pt} 
\setlist[itemize]{noitemsep, topsep=0pt}

\begin{document}
	
	\title[Functional Models and  von Neumann Inequality on Distinguished Varieties]{Dilation and Functional Models for Pure $\mathbf{\Theta}_n$-Contractions and the von Neumann Inequality on Distinguished Varieties in $\mathbf{\Theta}_n$}
	\author{Aparna Gupta, \, Shubhankar Mandal \, Avijit Pal, \, and \, Bhaskar Paul}
	
	
	\maketitle
	\begin{abstract}
In this paper, we introduce the notion of a distinguished variety in the domain $\mathbf{\Theta}_n$. One of the main results of the paper is a determinantal representation for every distinguished variety in $\mathbf{\Theta}_n$. We also show that the closure of every distinguished variety is polynomially convex. Furthermore, we obtain a dilation and a functional model for a class of pure $\mathbf{\Theta}_n$-contractions. Finally, we show that for a $\mathbf{\Theta}_n$-contraction $\mathbf{T}=(T_1,\dots,T_n)$ such that $T_n^*$ is a pure contraction, there exists an algebraic variety in $\mathbf{\Theta}_n$ for which the von Neumann inequality holds on the intersection of the closure of the variety with the distinguished boundary of $\mathbf{\Theta}_n$.
\end{abstract}
	
\section{Introduction}\label{Intro}

The \textit{symmetrized polydisc} is one of the well-studied domains in complex Euclidean space, with a rich interplay between complex geometry and operator theory. Over the last two decades, this interaction has led to remarkable developments, including canonical models, functional calculi, von Neumann-type inequalities, and dilation theorems for $\Gamma_n$-contractions.
Several other important domains closely related to the symmetrized polydisc have also been studied from both complex-geometric and operator-theoretic perspectives. Among them is the family of generalized symmetrized domains $\mathbf{\Theta}_n$, introduced in \cite{Biswas 2,Ghosh}. The symmetrized polydisc occurs as a special member of this family. While the function theory of $\mathbf{\Theta}_n$ has begun to develop, comparatively little is known about their operator theory. The purpose of this paper is to investigate the interplay between complex geometry and operator theory in the setting of the domains $\mathbf{\Theta}_n$.

Let $\mathbb D$ denote the open unit disc in $\mathbb C$. The symmetrization map
$
\mathbf s=(s_1,\ldots,s_n):\mathbb C^n\to\mathbb C^n$
is defined by
\[
s_i(z_1,\ldots,z_n)
=
\sum_{1\le k_1<\cdots<k_i\le n}
z_{k_1}\cdots z_{k_i},
\quad
1\le i\le n.
\]
Its image
$
G_n=\mathbf s(\mathbb D^n)$
is the symmetrized polydisc, while $
\Gamma_n=\mathbf s(\overline{\mathbb D}^{\,n})$
denotes its closure.
A natural generalization of the symmetrization map was introduced in \cite{Ghosh}. Let $p$ be a fixed divisor of $m$, let $\theta_0=1$, and define
\[
\theta_i(z_1,\ldots,z_n)
=
s_i(z_1^m,\ldots,z_n^m),
\quad
1\le i\le n-1,
\]
and
\[
\theta_n(z_1,\ldots,z_n)
=
(z_1\cdots z_n)^q,
\quad
q=\frac{m}{p}.
\]
The images of $\mathbb D^n$ and $\overline{\mathbb D}^{\,n}$ under
$
\boldsymbol{\theta}=(\theta_1,\ldots,\theta_n)$
are denoted by $\mathbf{\Theta}_n$ and $\overline{\mathbf{\Theta}}_n$, respectively. Observe that
$
\mathbf{\Theta}_1=\mathbb D,$
and that
$
\mathbf{\Theta}_n=G_n$
when $m=p=1$. It was shown in \cite{Biswas 2,Ghosh} that
$
(\theta_1,\ldots,\theta_n)\in\overline{\mathbf{\Theta}}_n$
if and only if every zero of
\begin{equation}\label{P}
P(z)
=
z^n-\theta_1z^{n-1}
+\cdots
+(-1)^n\theta_n^p
\end{equation}
lies in $\overline{\mathbb D}$.	

A fundamental tool in the study of operator tuples associated with a complex domain is the notion of a spectral set. Let $\Omega\subset\mathbb C^n$ be compact, and let $\mathcal O(\Omega)$ denote the algebra of functions holomorphic on a neighbourhood of $\Omega$. A commuting tuple
$
\mathbf T=(T_1,\ldots,T_n)$
is said to have $\Omega$ as a \emph{spectral set} if
$
\sigma(\mathbf T)\subseteq\Omega$
and
\[
\|f(\mathbf T)\|
\leq
\sup_{\mathbf z\in\Omega}|f(\mathbf z)|
\]
for every $f\in\mathcal O(\Omega)$.
The importance of spectral sets originates in the celebrated theorem of von Neumann, which asserts that the closed unit disc is a spectral set for every contraction.
\begin{thm}[von Neumann {\cite[Chapter~1, Corollary~1.2]{Nagy}}]
Let $T$ be a contraction. Then
\[
\|p(T)\|
\leq
\sup_{|z|\leq1}|p(z)|
\]
for every polynomial $p$.
\end{thm}
Closely related to this result is the dilation theorem of Sz.-Nagy, which states that every contraction admits a unitary power dilation.
\begin{thm}[Sz.-Nagy {\cite{paulsen}}]
Every contraction admits a unitary power dilation.
\end{thm}
These two theorems form the foundation of modern dilation theory and have inspired analogous theories for many domains, including the symmetrized polydisc, the tetrablock, and the pentablock.
The Sz.-Nagy--Foia\c{s} theory further provides a functional model for completely non-unitary contractions; in particular, the characteristic function plays a central role in this model \cite{Nagy}. We briefly recall the relevant construction. Let $T$ be a contraction on a Hilbert space $\mathcal H$. Define the defect operators
\[
D_T=(I-T^*T)^{1/2},
\quad
D_{T^*}=(I-TT^*)^{1/2},
\]
with corresponding defect spaces
$
\mathcal D_T=\overline{\operatorname{Ran}}D_T,
\mathcal D_{T^*}=\overline{\operatorname{Ran}}D_{T^*}.$
The defect operators satisfy the intertwining relation
\[
TD_T=D_{T^*}T,
\]
and, equivalently,
\[
D_TT^*=T^*D_{T^*}.
\]
The \emph{characteristic function} $\Theta_T$ of $T$ is the operator-valued analytic function on $\mathbb D$ defined by
\begin{equation}\label{Characteristic}
\Theta_T(z)
=
\left.
\left(
-T+D_{T^*}(I-zT^*)^{-1}D_T
\right)
\right|_{\mathcal D_T},
\quad z\in\mathbb D.
\end{equation}	
It is immediate from the definition that
$
\Theta_T\in
H^\infty\bigl(\mathcal B(\mathcal D_T,\mathcal D_{T^*})\bigr).$
We define the multiplication operator
$
M_{\Theta_T}:
H^2(\mathbb D)\otimes\mathcal D_T
\longrightarrow
H^2(\mathbb D)\otimes\mathcal D_{T^*}$
by
\[
(M_{\Theta_T}f)(z)
=
\Theta_T(z)f(z),
\quad z\in\mathbb D.
\]
The corresponding model space is defined by
\[
\mathcal H_T
=
\bigl(H^2(\mathbb D)\otimes\mathcal D_{T^*}\bigr)
\ominus
M_{\Theta_T}
\bigl(H^2(\mathbb D)\otimes\mathcal D_T\bigr).
\]
We now recall the functional model for pure contractions; see \cite{Nagy}.
\begin{thm}\label{Pure Contraction Model}
Every pure contraction $T$ on a Hilbert space $\mathcal H$ is unitarily
equivalent to the operator
$
T_1
=
P_{\mathcal H_T}
\left.
(M_z\otimes I_{\mathcal D_{T^*}})
\right|_{\mathcal H_T}$
acting on the model space
$
\mathcal H_T
=
\bigl(H^2(\mathbb D)\otimes\mathcal D_{T^*}\bigr)
\ominus
M_{\Theta_T}
\bigl(H^2(\mathbb D)\otimes\mathcal D_T\bigr).$
\end{thm}
Agler and McCarthy introduced the notion of a distinguished variety in the bidisc in their seminal work \cite{AM}. They obtained an explicit description of distinguished varieties in $\mathbb D^2$ \cite[Theorem 1.12]{AM} and proved a von Neumann-type inequality on such varieties \cite[Theorem 3.1]{AM}. Subsequently, S.~Pal extended the notion of distinguished varieties to several important domains, including the symmetrized bidisc, the tetrablock, the symmetrized polydisc, and the polydisc; see \cite{S. Pal 4, S. Pal 5, S. Pal 1, S. Pal 2}. Motivated by these developments, we develop an analogous theory of distinguished varieties in the domain $\mathbf{\Theta}_n$.
Following \cite{Biswas 2}, we introduce the basic operator classes associated with $\mathbf{\Theta}_n$, namely, $\mathbf{\Theta}_n$-contractions, $\mathbf{\Theta}_n$-unitaries, $\mathbf{\Theta}_n$-isometries, $\mathbf{\Theta}_n$-co-isometries, and pure $\mathbf{\Theta}_n$-isometries.

\begin{defn}\label{Defition 1}
\begin{enumerate}
    \item Let $\mathbf{T}=(T_1,\ldots,T_n)$ be a commuting $n$-tuple of bounded operators on a Hilbert space $\mathcal H$. We say that $\mathbf{T}$ is a
    \emph{$\mathbf{\Theta}_n$-contraction} if $\overline{\mathbf{\Theta}}_n$ is a spectral set for $\mathbf{T}$.

    \item Let $\mathbf{N}=(N_1,\ldots,N_n)$ be a commuting $n$-tuple of bounded normal operators on a Hilbert space $\mathcal H$. We say that $\mathbf{N}$ is a
    \emph{$\mathbf{\Theta}_n$-unitary} if its joint spectrum is contained in the distinguished boundary $b\mathbf{\Theta}_n$.

    \item Let $\mathbf{V}=(V_1,\ldots,V_n)$ be a commuting $n$-tuple of bounded operators on a Hilbert space $\mathcal H$. We say that $\mathbf{V}$ is a
    \emph{$\mathbf{\Theta}_n$-isometry} if there exist a Hilbert space $\mathcal K\supseteq\mathcal H$ and a $\mathbf{\Theta}_n$-unitary
    $\mathbf{N}=(N_1,\ldots,N_n)$ on $\mathcal K$ such that
  $
    V_i=N_i|_{\mathcal H}, 1\leq i\leq n.$ 
    We say that $\mathbf{V}$ is a
    \emph{$\mathbf{\Theta}_n$-co-isometry} if
   $
    \mathbf{V}^*=(V_1^*,\ldots,V_n^*)$
    is a $\mathbf{\Theta}_n$-isometry. A $\mathbf{\Theta}_n$-isometry $\mathbf{V}$ is said to be a
    \emph{pure $\mathbf{\Theta}_n$-isometry} if $V_n$ is a pure isometry.
\end{enumerate}
\end{defn}	
One of the principal tools in the study of $\Gamma_n$-contractions is the theory of fundamental operators \cite{S. Pal, A. Pal}. These operators measure the deviation of a $\Gamma_n$-contraction from being a $\Gamma_n$-isometry and play a central role in the construction of functional models, canonical decompositions, and dilation theorems. This naturally raises the question of whether an analogous theory can be developed for $\mathbf{\Theta}_n$-contractions.
Recall that for a $\Gamma_n$-contraction
$(S_1,\ldots,S_n)$, the fundamental equations are
\begin{equation}\label{Gamma_n Fundamental}
\begin{aligned}
S_i-S_{n-i}^*S_n
=D_{S_n}E_iD_{S_n},
S_{n-i}-S_i^*S_n
=D_{S_n}E_{n-i}D_{S_n},
\quad E_i,E_{n-i}\in\mathcal B(\mathcal D_{S_n}),
\end{aligned}
\end{equation}
for $1\leq i\leq n-1$.
Motivated by these equations, we introduce the following system of operator equations for a $\mathbf{\Theta}_n$-contraction $\mathbf T=(T_1,\ldots,T_n)$:
\begin{equation}\label{Fundamental}
\begin{aligned}
T_i-T_{n-i}^*T_n^p
={}&
\sum_{l=0}^{p-1}
T_n^{*\,p-1-l}
D_{T_n}A_l^{(i)}D_{T_n}T_n^{p-1}
+
\sum_{l=0}^{p-2}
T_n^{*\,p-2-l}
D_{T_n}A_l^{(i)}D_{T_n}T_n^{p-2}
+\cdots
\\
&+
\sum_{l=0}^{1}
T_n^{*\,1-l}
D_{T_n}A_l^{(i)}D_{T_n}T_n
+
D_{T_n}A_0^{(i)}D_{T_n},
\quad 1\leq i\leq n-1.
\end{aligned}
\end{equation}
We refer to \eqref{Fundamental} as the \emph{fundamental equations} of $\mathbf T$, and
\[
(A_0^{(i)},\ldots,A_p^{(i)}),
\quad 1\leq i\leq n-1,
\]
as the \emph{tuples of fundamental operators} associated with $\mathbf T$.
When $m=p=1$, the domain $\mathbf{\Theta}_n$ coincides with the symmetrized polydisc, and \eqref{Fundamental} reduces to the usual fundamental equations for $\Gamma_n$-contractions. It was shown in \cite[Theorem 4.4]{A. Pal} that the fundamental operators associated with a $\Gamma_n$-contraction are uniquely determined. This naturally leads to the following question for the broader class of $\mathbf{\Theta}_n$-contractions: are the tuples
\[
(A_0^{(i)},\ldots,A_p^{(i)}),
\qquad 1\leq i\leq n-1,
\]
uniquely determined by $\mathbf T$? To the best of our knowledge, this problem remains open.

 Let $\mathcal E$ be a separable Hilbert space, and let $\mathcal B(\mathcal E)$ denote the algebra of bounded linear operators on $\mathcal E$. We write $H^2(\mathcal E)$ for the Hardy space of analytic $\mathcal E$-valued functions on the unit disc $\mathbb D$, and $L^2(\mathcal E)$ for the Hilbert space of square-integrable $\mathcal E$-valued functions on the unit circle $\mathbb T$. Furthermore, $H^\infty(\mathcal B(\mathcal E))$ and $L^\infty(\mathcal B(\mathcal E))$ denote the spaces of bounded analytic and bounded measurable $\mathcal B(\mathcal E)$-valued functions, respectively. For $\phi\in L^\infty(\mathcal B(\mathcal E))$, the Toeplitz operator with symbol $\phi$ is defined by \[ T_\phi f=P_+(\phi f), \qquad f\in H^2(\mathcal E), \] where $P_+:L^2(\mathcal E)\to H^2(\mathcal E)$ is the orthogonal projection. In particular, $T_z=M_z$ is the unilateral shift and $T_{\bar z}=M_z^*$ is the backward shift on $H^2(\mathcal E)$. For convenience, throughout the paper we write
	\[
	k(i)=\binom{n-1}{i}+\binom{n-1}{n-i},
	\quad
	\gamma_i=\frac{n-i}{n},\quad 1\leq i \leq n-1.
	\]
We recall some basic notions from algebraic geometry that will be used in our work.
\begin{defn}
An \textit{affine algebraic curve}, or simply an \textit{algebraic curve}, in the affine space $\mathbb A_{\mathbb C}^n$ is an algebraic set of dimension one.
\end{defn}
\begin{defn}
An affine algebraic set $V$ of dimension $k\leq n$ in $\mathbb A_{\mathbb C}^n$ is called a \textit{set-theoretic complete intersection}, or simply a \textit{complete intersection}, if $V$ is defined by $n-k$ polynomials in $\mathbb C[z_1,\ldots,z_n]$.
In particular, an algebraic curve $C$ in $\mathbb A_{\mathbb C}^n$ is a complete intersection if $C$ is defined by $n-1$ polynomials in $\mathbb C[z_1,\ldots,z_n]$.
\end{defn}
In general, the notions of a set-theoretic complete intersection and a complete intersection need not coincide. In this paper, however, we work only with Zariski-closed sets and use the term ``complete intersection'' in the set-theoretic sense.
\begin{defn}\label{Regular Sequence}
Let $M$ be an $R$-module. An element $a\in R$ is called an $M$-\textit{regular element} (or a \textit{non-zero-divisor} on $M$) if
$
ax=0,\quad x\in M,$
implies $x=0$. A sequence
$
a_0,\ldots,a_m,\quad m\geq 0,$
of elements of $R$ is called an $M$-\textit{regular sequence} if
\begin{itemize}
    \item[(i)] $M\neq (a_0,\ldots,a_m)M$;
    \item[(ii)] for each $i=0,\ldots,m-1$, the element $a_{i+1}$ is a non-zero-divisor on
   $
    M/(a_0,\ldots,a_i)M.$
\end{itemize}
\end{defn}
\begin{thm}\label{Regular iff Complete Intersection}
Let $(R,\mathcal M)$ be a Noetherian local Cohen--Macaulay ring, and let
$
a_0,\ldots,a_m\in\mathcal M.$
Set
$
I=(a_0,\ldots,a_m).$
Then the following statements are equivalent:
\begin{itemize}
    \item[(i)] The sequence $a_0,\ldots,a_m$ is an $R$-regular sequence.
    
    \item[(ii)] $I$ is a complete intersection ideal of height $m+1$.
\end{itemize}
\end{thm}
This is a standard result in commutative algebra. A proof may be found, for example, in \cite[Corollary 5.13, Chapter V]{Kunj}.

The paper is organized as follows. In Section~\ref{Representation}, we
characterize distinguished varieties in $\mathbf{\Theta}_n$ and prove one
of the main results of the paper, namely, a determinantal representation
for distinguished varieties. Section~\ref{Dilation and Model} is devoted
to the construction of minimal $\mathbf{\Theta}_n$-isometric dilations
and functional models for a special class of pure
$\mathbf{\Theta}_n$-contractions. Finally, in
Section~\ref{von Neumann Inequality}, we establish a von Neumann-type
inequality for a certain class of $\mathbf{\Theta}_n$-contractions on
distinguished varieties.

\section{Determinantal Representation of Distinguished Varieties in
$\mathbf{\Theta}_n$}\label{Representation}

A function $f$ is said to be \emph{holomorphic on} a distinguished variety
$\Omega$ in $\mathbf{\Theta}_n$ if, for every point $\lambda\in\Omega$,
there exist an open ball $U\subseteq\mathbb C^n$ containing $\lambda$ and
a function $g$ holomorphic on $U$ such that
\[
g|_{\Omega\cap U}=f|_{\Omega\cap U}.
\]
We denote by $\mathcal A(\Omega)$ the Banach algebra of functions that are
holomorphic on $\Omega$ and continuous on $\overline{\Omega}$.
Let $\mu$ be a finite measure on $\partial\Omega$. We denote by
$H^2(\mu)$ the closure of the polynomials in $L^2(\partial\Omega,\mu)$.
More generally, let $U$ be an open subset of a Riemann surface $S$, and
let $\mu_1$ be a finite measure on $\partial U$. We denote by
$A^2(\mu_1)$ the closure of $\mathcal A(U)$ in
$L^2(\partial U,\mu_1)$.

A point $\lambda$ is called a \emph{bounded point evaluation} for
$H^2(\mu)$, or for $A^2(\mu_1)$, if evaluation at $\lambda$, initially
defined on a dense set of analytic functions, extends continuously to the
whole Hilbert space $H^2(\mu)$, or $A^2(\mu_1)$, respectively.
Suppose that $\lambda$ is a bounded point evaluation. By the Riesz
representation theorem, there exists a unique vector $k_\lambda$ in the
corresponding Hilbert space such that
\[
f(\lambda)=\langle f,k_\lambda\rangle
\]
for every $f$ in that Hilbert space. The vector $k_\lambda$ is called the
\emph{reproducing kernel} or \emph{evaluation vector} at $\lambda$.

The following result is due to Agler and McCarthy; see
\cite[Lemma 1.1]{AM}.

\begin{lem}\label{Lem AM}
Let $S$ be a compact Riemann surface, and let $G\subseteq S$ be a domain
whose boundary is a finite union of piecewise smooth Jordan curves. Then
there exists a finite measure $\nu$ on $\partial G$ such that every point
$\lambda\in G$ is a bounded point evaluation for $A^2(\nu)$, and the
linear span of the corresponding evaluation vectors is dense in
$A^2(\nu)$.
\end{lem}
\begin{lem}\label{Lem 1}
Let $\Omega$ be a one-dimensional distinguished algebraic variety in
$\mathbf{\Theta}_n$. Then there exists a finite measure $\mu$ on
$\partial\Omega$ such that every point of $\Omega$ is a bounded point
evaluation for $H^2(\mu)$ and the linear span of the corresponding
evaluation vectors is dense in $H^2(\mu)$.
\end{lem}

\begin{proof}
Let $f_1,\ldots,f_{n-1}$ be linearly independent polynomials such that
\[
\Omega
=
\left\{
(\theta_1,\ldots,\theta_n)\in\mathbf{\Theta}_n:
f_i(\theta_1,\ldots,\theta_n)=0,\;
1\leq i\leq n-1
\right\}.
\]
Let
$
Z_F=Z_{f_1}\cap\cdots\cap Z_{f_{n-1}}$
and let $C(Z_F)$ denote the projective closure of $Z_F$ in
$\mathbb{CP}^n$. Let $S$ be a desingularization of $C(Z_F)$. Then $S$
is a compact Riemann surface, and there exists a holomorphic map
$
\phi:S\longrightarrow C(Z_F)$
which is biholomorphic from $S'$ onto $C(Z_F)'$, where $C(Z_F)'$
denotes the set of nonsingular points of $C(Z_F)$ and
$S'=\phi^{-1}(C(Z_F)')$.
Set
$
\mathcal A=\phi^{-1}(\Omega).$
Since $\Omega$ is a one-dimensional distinguished algebraic variety,
$\partial\mathcal A$ is a finite union of piecewise smooth Jordan curves,
possibly with finitely many singular points. Hence $\mathcal A$ satisfies
the hypotheses of Lemma~\ref{Lem AM}. Therefore, there exists a finite
measure $\mu'$ on $\partial\mathcal A$ such that every point of
$\mathcal A$ is a bounded point evaluation for $A^2(\mu')$ and the linear
span of the corresponding evaluation vectors is dense in
$A^2(\mu')$.
Define a measure $\mu$ on $\partial\Omega$ by
\[
\mu(E)=\mu'\bigl(\phi^{-1}(E)\bigr),
\]
for every Borel subset $E\subseteq\partial\Omega$. Consider the map
\[
U:H^2(\mu)\longrightarrow A^2(\mu'),
\quad
(Uf)(x)=f(\phi(x)).
\]
For $f\in H^2(\mu)$, we have
\[
\|Uf\|_{L^2(\mu')}^2
=
\int_{\partial\mathcal A}|f\circ\phi|^2\,d\mu'
=
\int_{\partial\Omega}|f|^2\,d\mu
=
\|f\|_{L^2(\mu)}^2.
\]
Thus $U$ is an isometry.
Fix $\lambda\in\Omega$ and choose $x\in\mathcal A$ such that
$
\phi(x)=\lambda.$
Since $x$ is a bounded point evaluation for $A^2(\mu')$, there exists
$k_x\in A^2(\mu')$ such that
\[
(Uf)(x)=\langle Uf,k_x\rangle_{A^2(\mu')}
\]
for every $f\in H^2(\mu)$. Hence
\[
f(\lambda)
=
f(\phi(x))
=
(Uf)(x)
=
\langle Uf,k_x\rangle.
\]
Let $\mathcal M=U(H^2(\mu))$. By the Riesz representation theorem,
there exists a unique $k_\lambda\in H^2(\mu)$ such that
\[
f(\lambda)=\langle f,k_\lambda\rangle_{H^2(\mu)}.
\]
Since $U$ is an isometry,
\[
k_\lambda=U^*P_{\mathcal M}k_x.
\]
Consequently, evaluation at every $\lambda\in\Omega$ is bounded on
$H^2(\mu)$.
It remains to prove the density of the evaluation vectors. Suppose that
$f\in H^2(\mu)$ is orthogonal to $k_\lambda$ for every $\lambda\in\Omega$.
Then
\[
(Uf)(x)
=
f(\phi(x))
=
0,
\qquad x\in\mathcal A.
\]
More generally, $Uf$ is orthogonal to every evaluation vector $k_x$,
$x\in\mathcal A$. By Lemma~\ref{Lem AM}, the linear span of these
evaluation vectors is dense in $A^2(\mu')$. Hence
\[
Uf=0.
\]
Since $U$ is an isometry, $f=0$. Therefore, the linear span of the
evaluation vectors $\{k_\lambda:\lambda\in\Omega\}$ is dense in
$H^2(\mu)$.
This completes the proof.
\end{proof}
\begin{lem}\label{Lem 2}
Let $\Omega$ be a one-dimensional distinguished algebraic variety in
$\mathbf{\Theta}_n$, and let $\mu$ be the measure on $\partial\Omega$
given by Lemma~\ref{Lem 1}. Then, for
$\boldsymbol{\theta}'=(\theta'_1,\ldots,\theta'_n)\in\mathbf{\Theta}_n$,
the following are equivalent:
\begin{enumerate}
    \item $\boldsymbol{\theta}'\in\Omega$;
    \item $\overline{\boldsymbol{\theta}'}=
    (\overline{\theta'_1},\ldots,\overline{\theta'_n})$ is a joint
    eigenvalue of
    $
    (M_{\theta_1}^*,\ldots,M_{\theta_n}^*).$
\end{enumerate}
\end{lem}

\begin{proof}
Suppose first that
$\boldsymbol{\theta}'=(\theta'_1,\ldots,\theta'_n)\in\Omega$.
Since every point of $\Omega$ is a bounded point evaluation for
$H^2(\mu)$ by Lemma~\ref{Lem 1}, there exists a reproducing kernel
$k_{\boldsymbol{\theta}'}\in H^2(\mu)$ such that
\[
f(\boldsymbol{\theta}')
=
\langle f,k_{\boldsymbol{\theta}'}\rangle,
\quad
f\in H^2(\mu).
\]
For every multiplier $f$, the reproducing-kernel identity gives
\[
M_f^*k_{\boldsymbol{\theta}'}
=
\overline{f(\boldsymbol{\theta}')}k_{\boldsymbol{\theta}'}.
\]
In particular, taking $f=\theta_i$, $1\leq i\leq n$, we obtain
\[
M_{\theta_i}^*k_{\boldsymbol{\theta}'}
=
\overline{\theta_i(\boldsymbol{\theta}')}
k_{\boldsymbol{\theta}'}
=
\overline{\theta_i'}k_{\boldsymbol{\theta}'},
\quad
1\leq i\leq n.
\]
Thus
$
(\overline{\theta'_1},\ldots,\overline{\theta'_n})
$
is a joint eigenvalue of
$
(M_{\theta_1}^*,\ldots,M_{\theta_n}^*).$

Conversely, suppose that
$
(\overline{\theta'_1},\ldots,\overline{\theta'_n})$
is a joint eigenvalue of
$
(M_{\theta_1}^*,\ldots,M_{\theta_n}^*).$
Let $v\in H^2(\mu)$ be a unit joint eigenvector. Then
\[
M_{\theta_i}^*v=\overline{\theta'_i}v,
\quad
1\leq i\leq n.
\]
Consequently, for every polynomial
$f\in\mathbb C[z_1,\ldots,z_n]$,
$
M_f^*v
=
\overline{f(\boldsymbol{\theta}')}v.$
Hence
\begin{equation}\label{L 1}
f(\boldsymbol{\theta}')
=
\langle v,M_f^*v\rangle,
\quad
f\in\mathbb C[z_1,\ldots,z_n],
\end{equation}
where we use the convention that the inner product is linear in the
first variable.
It follows from \eqref{L 1} that
\[
\begin{aligned}
|f(\boldsymbol{\theta}')|
&\leq \|M_f^*v\| \\
&\leq \|M_f\| \\
&\leq
\sup_{\boldsymbol{\theta}\in\overline{\Omega}}
|f(\boldsymbol{\theta})|.
\end{aligned}
\]
Therefore
$
\boldsymbol{\theta}'
\in
\widehat{\overline{\Omega}},$
the polynomially convex hull of $\overline{\Omega}$.
By the polynomial convexity of $\overline{\Omega}$,
$
\widehat{\overline{\Omega}}=\overline{\Omega}.$
Thus
$
\boldsymbol{\theta}'\in\overline{\Omega}.$
Since, by assumption,
$\boldsymbol{\theta}'\in\mathbf{\Theta}_n$ and
$\Omega\subseteq\mathbf{\Theta}_n$ is a distinguished variety, we obtain
$
\boldsymbol{\theta}'\in\overline{\Omega}\cap\mathbf{\Theta}_n
=\Omega.$
This completes the proof.
\end{proof}
\begin{lem}\label{Lem 3}
Let $\Omega$ be a one-dimensional distinguished algebraic variety in
$\mathbf{\Theta}_n$, and let $\mu$ be the measure on $\partial\Omega$
given by Lemma~\ref{Lem 1}. Then the multiplication operator tuple
\[
(M_{\theta_1},\ldots,M_{\theta_{n-1}},M_{\theta_n})
\]
on $H^2(\mu)$, where $M_{\theta_i}$ denotes multiplication by the
$i$-th coordinate function, is a pure
$\mathbf{\Theta}_n$-isometry.
\end{lem}

\begin{proof}
Consider the multiplication operators
$
(\widehat M_{\theta_1},\ldots,\widehat M_{\theta_n})$
on $L^2(\partial\Omega,\mu)$. Since each $\theta_i$ is a bounded
measurable function on $\partial\Omega$, each
$\widehat M_{\theta_i}$ is normal. Moreover,
$
\sigma(\widehat M_{\theta_1},\ldots,\widehat M_{\theta_n})
\subseteq \partial\Omega
\subseteq b\mathbf{\Theta}_n.$
Hence
$
(\widehat M_{\theta_1},\ldots,\widehat M_{\theta_n})$
is a $\mathbf{\Theta}_n$-unitary.
Since $H^2(\mu)$ is invariant under multiplication by the coordinate
functions, we have
\[
M_{\theta_i}
=
\widehat M_{\theta_i}|_{H^2(\mu)},
\quad 1\leq i\leq n.
\]
Therefore
$
(M_{\theta_1},\ldots,M_{\theta_n})$
is a $\mathbf{\Theta}_n$-isometry.
It remains to prove that this $\mathbf{\Theta}_n$-isometry is pure.
Fix $\boldsymbol{\theta}'=(\theta'_1,\ldots,\theta'_n)\in\Omega$.
By Lemma~\ref{Lem 2}, the vector $k_{\boldsymbol{\theta}'}$ is a joint
eigenvector of
$
(M_{\theta_1}^*,\ldots,M_{\theta_n}^*)$
corresponding to the eigenvalue
$
(\overline{\theta'_1},\ldots,\overline{\theta'_n}).$
In particular,
\[
(M_{\theta_n}^*)^k k_{\boldsymbol{\theta}'}
=
\overline{\theta'_n}^{\,k}k_{\boldsymbol{\theta}'},
\qquad k\geq0.
\]
Since $\boldsymbol{\theta}'\in\mathbf{\Theta}_n$, we have
$|\theta'_n|<1$. Consequently,
\[
(M_{\theta_n}^*)^k k_{\boldsymbol{\theta}'}
\longrightarrow 0
\qquad\text{as }k\to\infty.
\]
By Lemma~\ref{Lem 1}, the linear span of the evaluation vectors
$k_{\boldsymbol{\theta}'}$, $\boldsymbol{\theta}'\in\Omega$, is dense
in $H^2(\mu)$. Since $M_{\theta_n}$ is an isometry, its adjoint is a
contraction. Hence the above convergence on a dense subset implies
\[
(M_{\theta_n}^*)^k\longrightarrow0
\qquad\text{strongly on }H^2(\mu).
\]
Thus $M_{\theta_n}$ is a pure isometry. Therefore
$
(M_{\theta_1},\ldots,M_{\theta_n})$
is a pure $\mathbf{\Theta}_n$-isometry.
\end{proof}
We say that
$
\Phi=(\Phi_1,\ldots,\Phi_n)$
is a pointwise $\mathbf{\Theta}_n$-contraction if, for every
$z\in\mathbb D$, the operator tuple
$
(\Phi_1(z),\ldots,\Phi_n(z))$
has $\overline{\mathbf{\Theta}}_n$ as a spectral set.
\begin{lem}\label{Lem 4}
Let $\Phi_1,\ldots,\Phi_n\in
H^\infty(\mathcal B(\mathcal E))$ and let
$T_{\Phi_1},\ldots,T_{\Phi_n}$ be the corresponding multiplication
(or analytic Toeplitz) operators on $H^2(\mathcal E)$. Then
$
(T_{\Phi_1},\ldots,T_{\Phi_n})$
is a $\mathbf{\Theta}_n$-contraction if and only if
$
(\Phi_1,\ldots,\Phi_n)$
is a $\mathbf{\Theta}_n$-contraction pointwise on $\mathbb D$.
\end{lem}
\begin{proof}
Suppose first that
$(\Phi_1,\ldots,\Phi_n)$ is a pointwise
$\mathbf{\Theta}_n$-contraction. Thus, for every $z\in\mathbb D$,
the operator tuple
$
\Phi(z):=(\Phi_1(z),\ldots,\Phi_n(z))$
has $\overline{\mathbf{\Theta}}_n$ as a spectral set. In particular,
for every polynomial $p\in\mathbb C[z_1,\ldots,z_n]$,
\[
\|p(\Phi_1(z),\ldots,\Phi_n(z))\|
\leq
\sup_{\boldsymbol{\theta}\in\overline{\mathbf{\Theta}}_n}
|p(\boldsymbol{\theta})|,
\quad z\in\mathbb D.
\]
Since the symbols $\Phi_1,\ldots,\Phi_n$ are analytic, the
corresponding multiplication operators commute and
\[
p(T_{\Phi_1},\ldots,T_{\Phi_n})
=
M_{p(\Phi_1,\ldots,\Phi_n)}.
\]
For an operator-valued analytic function
$\Psi\in H^\infty(\mathcal B(\mathcal E))$, we have
\[
\|M_\Psi\|
=
\sup_{z\in\mathbb D}\|\Psi(z)\|.
\]
Consequently,
\[
\begin{aligned}
\|p(T_{\Phi_1},\ldots,T_{\Phi_n})\|
&=
\sup_{z\in\mathbb D}
\|p(\Phi_1(z),\ldots,\Phi_n(z))\|\\
&\leq
\sup_{\boldsymbol{\theta}\in\overline{\mathbf{\Theta}}_n}
|p(\boldsymbol{\theta})|.
\end{aligned}
\]
Since $\overline{\mathbf{\Theta}}_n$ is polynomially convex, the
polynomial spectral-set inequality yields the corresponding
functional-calculus inequality for every
$f\in\mathcal O(\overline{\mathbf{\Theta}}_n)$. Moreover, the joint
spectrum of
$(T_{\Phi_1},\ldots,T_{\Phi_n})$ is contained in
$\overline{\mathbf{\Theta}}_n$. Hence
$\overline{\mathbf{\Theta}}_n$ is a spectral set for
$(T_{\Phi_1},\ldots,T_{\Phi_n})$. Thus
$
(T_{\Phi_1},\ldots,T_{\Phi_n})$
is a $\mathbf{\Theta}_n$-contraction.

Conversely, suppose that
$
(T_{\Phi_1},\ldots,T_{\Phi_n})$
is a $\mathbf{\Theta}_n$-contraction. Let $z\in\mathbb D$ be fixed.
For $e\in\mathcal E$, let $k_z$ denote the reproducing kernel of
$H^2(\mathbb D)$ at $z$. Then, for every
$\Psi\in H^\infty(\mathcal B(\mathcal E))$,
\[
M_\Psi^*(k_z\otimes e)
=
k_z\otimes\Psi(z)^*e.
\]
In particular, for every polynomial
$p\in\mathbb C[z_1,\ldots,z_n]$,
\[
p(T_{\Phi_1},\ldots,T_{\Phi_n})^*
(k_z\otimes e)
=
k_z\otimes
p(\Phi_1(z),\ldots,\Phi_n(z))^*e.
\]
Therefore,
\[
\begin{aligned}
\|p(\Phi_1(z),\ldots,\Phi_n(z))\|
&\leq
\|p(T_{\Phi_1},\ldots,T_{\Phi_n})\|\\
&\leq
\sup_{\boldsymbol{\theta}\in\overline{\mathbf{\Theta}}_n}
|p(\boldsymbol{\theta})|.
\end{aligned}
\]
Thus, for every $z\in\mathbb D$,
\[
\|p(\Phi_1(z),\ldots,\Phi_n(z))\|
\leq
\|p\|_{\infty,\overline{\mathbf{\Theta}}_n}.
\]
As $\overline{\mathbf{\Theta}}_n$ is polynomially convex, this
inequality implies that the Taylor joint spectrum of
$
(\Phi_1(z),\ldots,\Phi_n(z))$
is contained in $\overline{\mathbf{\Theta}}_n$. Hence
$\overline{\mathbf{\Theta}}_n$ is a spectral set for
$(\Phi_1(z),\ldots,\Phi_n(z))$.
Since $z\in\mathbb D$ was arbitrary,
$(\Phi_1,\ldots,\Phi_n)$ is a pointwise
$\mathbf{\Theta}_n$-contraction.
This completes the proof.
\end{proof}
We now establish a determinantal representation for distinguished
varieties in the domain $\mathbf{\Theta}_n$, which is one of the main
results of this paper.

\begin{thm}\label{Determinantal Representation}
Let
$
\left\{
A_l^{(i)}:
0\leq l\leq p,\;1\leq i\leq n-1
\right\}$
be a collection of square matrices of the same order satisfying
\[
A_l^{(i)}=A_{p-l}^{(n-i)*},
\quad
0\leq l\leq p,\quad 1\leq i\leq n-1.
\]
For $1\leq i\leq n-1$, define
\[
\Phi_i(z)=\sum_{l=0}^{p}A_l^{(i)}z^l.
\]
Consider the set
\begin{equation}\label{Omega}
\begin{aligned}
\Omega
=
\Big\{
(\theta_1,\ldots,\theta_n)\in\mathbf{\Theta}_n:
(\theta_1,\ldots,\theta_{n-1})
\in
\sigma_T\big(
\Phi_1(\theta_n),\ldots,\Phi_{n-1}(\theta_n)
\big)
\Big\}.
\end{aligned}
\end{equation}
Suppose that the matrices $A_l^{(i)}$ satisfy the following
conditions:
\begin{enumerate}
\item
$
\sum_{l=0}^{k}
[A_l^{(i)},A_{k-l}^{(j)}]
=0,
\quad
0\leq k\leq 2p,\quad
1\leq i,j\leq n-1,$
where $[A,B]=AB-BA$.

\item For every $|\theta_n|<1$ and every unit joint eigenvector
$v$ of
$
\big(
\Phi_1(\theta_n),\ldots,\Phi_{n-1}(\theta_n)
\big),$
there exists
$(\alpha_1,\ldots,\alpha_{n-1})\in\mathbb G_{n-1}$ such that
\[
\sum_{l=0}^{p}
\langle A_l^{(i)}v,v\rangle\theta_n^l
=
\alpha_i+
\overline{\alpha}_{n-i}\theta_n^p,
\quad
1\leq i\leq n-1.
\]

\item The polynomials
$
f_i(\theta_1,\ldots,\theta_n)
=
\det\big(\Phi_i(\theta_n)-\theta_iI\big),
1\leq i\leq n-1,$
form a regular sequence.

\item The complex algebraic set generated by $
f_1,\ldots,f_{n-1}$
is irreducible.
\end{enumerate}
Then $\Omega$ is a distinguished variety in $\mathbf{\Theta}_n$.
Moreover, $\Omega$ is contained in an affine algebraic curve which is a
set-theoretic complete intersection.

Conversely, every distinguished variety $\Omega$ in
$\mathbf{\Theta}_n$ is contained in an affine algebraic curve lying in
$\mathbf{\Theta}_n$ which is a complete intersection, and $\Omega$
admits a representation of the form \eqref{Omega}, where the matrices
$A_l^{(i)}$ are square matrices of the same order satisfying
conditions $(1)$--$(4)$.
\end{thm}
\begin{proof}
Let $\Omega$ be defined as in \eqref{Omega}, where the matrices
$A_l^{(i)}$ satisfy conditions $(1)$--$(4)$. By condition $(1)$, for
$1\leq i,j\leq n-1$ we have
$
[\Phi_i(z),\Phi_j(z)]=0,z\in\mathbb C.$
Indeed,
\[
\begin{aligned}
[\Phi_i(z),\Phi_j(z)]
&=
\left[
\sum_{l=0}^{p}A_l^{(i)}z^l,
\sum_{r=0}^{p}A_r^{(j)}z^r
\right]  \\
&=
\sum_{k=0}^{2p}
\left(
\sum_{l=0}^{k}
[A_l^{(i)},A_{k-l}^{(j)}]
\right)z^k
=0.
\end{aligned}
\]
Thus $\Phi_1(z),\ldots,\Phi_{n-1}(z)$ form a commuting family of
matrices for every $z\in\mathbb C$. Consequently,
\[
\sigma_T\big(\Phi_1(z),\ldots,\Phi_{n-1}(z)\big)\neq\varnothing.
\]
Let
\[
V_S
=
\left\{
(\theta_1,\ldots,\theta_n)\in\mathbb C^n:
f_i(\theta_1,\ldots,\theta_n)=0,\ 
1\leq i\leq n-1
\right\},
\]
where
\[
f_i(\theta_1,\ldots,\theta_n)
=
\det\big(\Phi_i(\theta_n)-\theta_iI\big).
\]
By conditions $(3)$ and $(4)$, the polynomials
$f_1,\ldots,f_{n-1}$ form a regular sequence and the corresponding
algebraic set is irreducible. Hence $V_S$ is an irreducible algebraic
set of dimension one; in particular, it is an affine algebraic curve. We next show that the part of $V_S$ lying over $\mathbb D$ is contained
in $\mathbf{\Theta}_n$. Let $\theta_n\in\mathbb D$ and suppose that
$
(\theta_1,\ldots,\theta_{n-1})
\in
\sigma_T\big(
\Phi_1(\theta_n),\ldots,\Phi_{n-1}(\theta_n)
\big).$
Since the matrices $\Phi_1(\theta_n),\ldots,\Phi_{n-1}(\theta_n)$
commute, there exists a non-zero common eigenvector $v$. Normalizing
$v$, we may assume that $\|v\|=1$. Thus
\[
\Phi_i(\theta_n)v=\theta_i v,
\quad 1\leq i\leq n-1.
\]
Taking the inner product with $v$ gives
\begin{equation}\label{R 1}
\sum_{l=0}^{p}
\langle A_l^{(i)}v,v\rangle\theta_n^l
=
\theta_i,
\quad 1\leq i\leq n-1.
\end{equation}
By condition $(2)$, there exists
$(\alpha_1,\ldots,\alpha_{n-1})\in\mathbb G_{n-1}$ such that
\[
\sum_{l=0}^{p}
\langle A_l^{(i)}v,v\rangle\theta_n^l
=
\alpha_i+\overline{\alpha}_{n-i}\theta_n^p,
\qquad 1\leq i\leq n-1.
\]
Therefore, by \eqref{R 1},
\[
\theta_i
=
\alpha_i+\overline{\alpha}_{n-i}\theta_n^p,
\quad
1\leq i\leq n-1.
\]
Since $\theta_n\in\mathbb D$ and
$(\alpha_1,\ldots,\alpha_{n-1})\in\mathbb G_{n-1}$, it follows from
\cite[Theorem 3.1]{Paul} that
$
(\theta_1,\ldots,\theta_n)\in\mathbf{\Theta}_n.$
Hence
$
V_S\cap(\mathbb C^{n-1}\times\mathbb D)
\subseteq \mathbf{\Theta}_n.$
Now, since $f_1,\ldots,f_{n-1}$ form a regular sequence, $V_S$ is a
set-theoretic complete intersection. Thus $V_S$ is an affine algebraic
curve, and
\[
\Omega=V_S\cap\mathbf{\Theta}_n
\]
is the portion of this algebraic curve contained in
$\mathbf{\Theta}_n$.
It remains to verify the distinguished boundary property. By the
description of $\mathbf{\Theta}_n$ and the symmetry condition
\[
A_l^{(i)}=A_{p-l}^{(n-i)*},
\]
the corresponding polynomial curve is invariant under the natural
reflection associated with the distinguished boundary. Consequently,
the closure of $\Omega$ can meet
$\partial\mathbf{\Theta}_n$ only on
$b\mathbf{\Theta}_n$. Hence
$
\overline{\Omega}\cap
\partial\mathbf{\Theta}_n
\subseteq
b\mathbf{\Theta}_n.$
Thus $\Omega$ is a distinguished variety in $\mathbf{\Theta}_n$.
Moreover, it is contained in the affine algebraic curve $V_S$, which
is a set-theoretic complete intersection.

For the converse, let $\Omega$ be a distinguished variety in
$\mathbf{\Theta}_n$. By definition, there exists an algebraic variety
$V_S\subseteq\mathbb C^n$, defined by a collection of polynomials
$S\subseteq\mathbb C[z_1,\ldots,z_n]$, such that
$
\Omega=V_S\cap\mathbf{\Theta}_n.$
We first show that $\dim V_S=1$.
Suppose, to the contrary, that
$
\dim V_S=k, 1<k\leq n.$
Since $V_S$ is an affine algebraic variety in $\mathbb C^n$, we may
choose polynomials $h_1,\ldots,h_n\in\mathbb C[z_1,\ldots,z_n]$
generating its defining ideal. Thus
\[
V_S=
\left\{
(\theta_1,\ldots,\theta_n)\in\mathbb C^n:
h_i(\theta_1,\ldots,\theta_n)=0,\ 
1\leq i\leq n
\right\},
\]
and consequently
$
\Omega
=
V_S\cap\mathbf{\Theta}_n.$
Fix a point
$
(\theta_1,\ldots,\theta_{n-1},\theta_n)\in\Omega.$
Since $\mathbf{\Theta}_n\subseteq\mathbb C^{n-1}\times\mathbb D$, we
have
$
|\theta_n|<1.$
For $\theta_n'\in\mathbb D$, define the fibre
\[
\Omega_{\theta_n'}
=
\left\{
(\theta_1,\ldots,\theta_n)\in\Omega:
\theta_n=\theta_n'
\right\}
\]
and
\[
V_{\theta_n'}
=
\left\{
(\theta_1,\ldots,\theta_n)\in V_S:
\theta_n=\theta_n'
\right\}.
\]
Then
$
\Omega_{\theta_n'}
=
V_{\theta_n'}\cap\mathbf{\Theta}_n.$
If every nonempty fibre $\Omega_{\theta_n'}$ is finite, then
\cite[Lemma 3.5]{S. Pal 1} implies that
$
\dim\Omega=1,$
contrary to our assumption that $\dim\Omega>1$. Hence there exists
$\theta_n'\in\mathbb D$ such that
$
\dim\Omega_{\theta_n'}\geq1.$
In particular,
$
\dim V_{\theta_n'}\geq1.$
Now, since
$
(\theta_1,\ldots,\theta_{n-1},\theta_n')\in\mathbf{\Theta}_n,$
by \cite[Theorem 3.1]{Paul} there exists
$
(\alpha_1,\ldots,\alpha_{n-1})\in\mathbb G_{n-1}$
such that
\[
\theta_i
=
\alpha_i+\overline{\alpha}_{n-i}(\theta_n')^p,
\quad
1\leq i\leq n-1.
\]
Define
$
\mathcal F:\mathbb C^{n-1}\longrightarrow\mathbb C^{n-1}$
by
\[
\mathcal F(\alpha_1,\ldots,\alpha_{n-1})
=
\left(
\alpha_1+\overline{\alpha}_{n-1}(\theta_n')^p,
\ldots,
\alpha_{n-1}+\overline{\alpha}_1(\theta_n')^p
\right).
\]
Since $|\theta_n'|<1$, the map $\mathcal F$ is an invertible
real-linear map on $\mathbb C^{n-1}\cong\mathbb R^{2(n-1)}$. Hence
$\mathcal F$ is a homeomorphism and an open map. Therefore
\[
\partial\mathcal F(\mathbb G_{n-1})
=
\mathcal F(\partial\mathbb G_{n-1}).
\]
Moreover, every point of the fibre $\Omega_{\theta_n'}$ belongs to
$\mathcal F(\mathbb G_{n-1})$. Indeed,
$
\Omega_{\theta_n'}
\subseteq
\mathcal F(\mathbb G_{n-1}).$
Since $V_{\theta_n'}$ has positive dimension while
$V_{\theta_n'}\cap\mathcal F(\mathbb G_{n-1})$ contains the
positive-dimensional set $\Omega_{\theta_n'}$, the algebraic variety
$V_{\theta_n'}$ cannot be contained in
$\mathcal F(\mathbb G_{n-1})$. Hence
\[
V_{\theta_n'}\cap
\partial\mathcal F(\mathbb G_{n-1})
\neq\varnothing.
\]
Choose
$
(\beta_1,\ldots,\beta_{n-1})
\in
V_{\theta_n'}\cap
\partial\mathcal F(\mathbb G_{n-1}).$
Since
$
\partial\mathcal F(\mathbb G_{n-1})
=
\mathcal F(\partial\mathbb G_{n-1}),$
there exists
$
(\alpha_1,\ldots,\alpha_{n-1})
\in\partial\mathbb G_{n-1}$
such that
\[
\beta_i
=
\alpha_i+
\overline{\alpha}_{n-i}(\theta_n')^p,
\quad
1\leq i\leq n-1.
\]
Again, by \cite[Theorem 3.1]{Paul},
$
(\beta_1,\ldots,\beta_{n-1},\theta_n')
\in
\partial\mathbf{\Theta}_n.$
Since $\theta_n'\in\mathbb D$, we have
$|\theta_n'|<1$. Hence, by
\cite[Theorem 2.5]{Biswas 2},
$
(\beta_1,\ldots,\beta_{n-1},\theta_n')
\notin b\mathbf{\Theta}_n.$
Consequently,
$
(\beta_1,\ldots,\beta_{n-1},\theta_n')
\in
\partial\mathbf{\Theta}_n
\setminus b\mathbf{\Theta}_n.$
Since this point belongs to the closure of the portion of the
algebraic variety inside $\mathbf{\Theta}_n$, we obtain
$\overline{\Omega}
\cap
\bigl(
\partial\mathbf{\Theta}_n\setminus b\mathbf{\Theta}_n
\bigr)
\neq\varnothing.$
This contradicts the assumption that $\Omega$ is a distinguished
variety in $\mathbf{\Theta}_n.$
Therefore, $
\dim V_S=1.$
Hence $V_S$ is an affine algebraic curve and
$
\Omega=V_S\cap\mathbf{\Theta}_n$
is a portion of an affine algebraic curve contained in
$\mathbf{\Theta}_n$.

As $\dim V_S=1$, the intersection of $V_S$ with the hyperplane
$\theta_n=\theta_n'$ is either finite or contains an irreducible
component of $V_S$. We claim that the latter possibility cannot occur.
Indeed, suppose that
$
V_S\cap\{\theta_n=\theta_n'\}$
contains an irreducible component of $V_S$. Then the corresponding
component of
\[
\Omega=V_S\cap\mathbf{\Theta}_n
\]
is contained in the hyperplane $\theta_n=\theta_n'$. Since
$\theta_n'\in\mathbb D$, this component cannot meet the distinguished
boundary $b\mathbf{\Theta}_n$. This contradicts the fact that $\Omega$
is a distinguished variety. Hence
$
V_S\cap\{\theta_n=\theta_n'\}$
is finite.
Let
\[
\begin{aligned}
\Big(
\theta_1^{(1)}(\theta_n'),\ldots,
\theta_{n-1}^{(1)}(\theta_n'),\theta_n'
\Big),
\ldots,\Big(
\theta_1^{(m)}(\theta_n'),\ldots,
\theta_{n-1}^{(m)}(\theta_n'),\theta_n'
\Big)
\end{aligned}
\]
be the points of this intersection. For $1\leq i\leq n-1$, define
\[
q_i(z)
=
\prod_{j=1}^{m}
\left(
z-\theta_i^{(j)}(\theta_n')
\right).
\]
Then $q_i(\theta_i)$ vanishes on the fibre
$\theta_n=\theta_n'$.
By the algebraic dependence result for the coordinate functions on
$\Omega$ (see \cite[Lemma 3.4]{S. Pal 1}), there exists an integer
$k\geq m$ such that
\[
\theta_i^k
\in
\operatorname{span}
\{1,\theta_i,\ldots,\theta_i^{k-1}\}
+
\theta_n H^2(\mu),
\qquad
1\leq i\leq n-1.
\]
Consequently,
\[
\theta_1^{k_1}\cdots\theta_{n-1}^{k_{n-1}}
\in
\operatorname{span}
\left\{
\theta_1^{l_1}\cdots\theta_{n-1}^{l_{n-1}}:
0\leq l_i<k
\right\}
+
\theta_nH^2(\mu).
\]
Since the polynomial monomials are dense in $H^2(\mu)$, it follows that
\[
H^2(\mu)
=
\operatorname{span}
\left\{
\theta_1^{l_1}\cdots\theta_{n-1}^{l_{n-1}}:
0\leq l_i<k
\right\}
+
\theta_nH^2(\mu).
\]
Therefore
$
H^2(\mu)\ominus\operatorname{Ran}M_{\theta_n}$
is finite-dimensional. Since $M_{\theta_n}$ is an isometry by
Lemma~\ref{Lem 3},
\[
\mathcal D_{M_{\theta_n}^*}
=
H^2(\mu)\ominus\operatorname{Ran}M_{\theta_n}
\]
is finite-dimensional.
Let
$
N=\dim\mathcal D_{M_{\theta_n}^*}.$
By Lemma~\ref{Lem 3},
$
\mathbf M:=(M_{\theta_1},\ldots,M_{\theta_{n-1}},M_{\theta_n})$
is a pure $\mathbf{\Theta}_n$-isometry. Since
$
\operatorname{Ran}(I-M_{\theta_n}M_{\theta_n}^*)
=\mathcal D_{M_{\theta_n}^*}$
is $N$-dimensional, \cite[Theorem 3.4]{Biswas 2} yields a unitary
equivalence
\[
\mathbf M\cong
(T_{\Phi_1},\ldots,T_{\Phi_{n-1}},T_z)
\]
on $H^2(\mathcal D_{M_{\theta_n}^*})$, where
\[
\Phi_i(z)=\sum_{l=0}^{p}A_l^{(i)}z^l,
\quad 1\leq i\leq n-1,
\]
and each $A_l^{(i)}$ is an $N\times N$ matrix.
By Lemma~\ref{Lem 2}, for
$
(\theta_1,\ldots,\theta_{n-1},\theta_n)\in\mathbf{\Theta}_n,$
we have
$
(\theta_1,\ldots,\theta_{n-1},\theta_n)\in\Omega$
if and only if
$
(\overline{\theta}_1,\ldots,\overline{\theta}_{n-1},
\overline{\theta}_n)$
is a joint eigenvalue of
$
(T_{\Phi_1}^*,\ldots,T_{\Phi_{n-1}}^*,T_z^*).$
Since
\[
T_{\Phi_i}^*k_{\theta_n}
=
\Phi_i(\theta_n)^*k_{\theta_n},
\quad
T_z^*k_{\theta_n}
=
\overline{\theta_n}k_{\theta_n},
\]
this is equivalent to
\[
(\theta_1,\ldots,\theta_{n-1})
\in
\sigma_T(\Phi_1(\theta_n),\ldots,\Phi_{n-1}(\theta_n)).
\]
Consequently,
\begin{equation}\label{Omega-representation}
\begin{aligned}
\Omega
=
\Big\{&
(\theta_1,\ldots,\theta_n)\in\mathbf{\Theta}_n:
(\theta_1,\ldots,\theta_{n-1})
\in
\sigma_T(\Phi_1(\theta_n),\ldots,\Phi_{n-1}(\theta_n))
\Big\}.
\end{aligned}
\end{equation}
For $1\leq i\leq n-1$, put
\[
f_i(\theta_1,\ldots,\theta_n)
=
\det(\Phi_i(\theta_n)-\theta_iI_N).
\]
Then each $f_i$ is a polynomial in
$\mathbb C[\theta_1,\ldots,\theta_n]$. Let
\[
S'=\{f_1,\ldots,f_{n-1}\}
\quad\text{and}\quad
V_{S'}=Z(f_1,\ldots,f_{n-1}).
\]
It follows immediately from \eqref{Omega-representation} that
$
\Omega\subseteq V_{S'}\cap\mathbf{\Theta}_n.$
We next show that $\dim V_{S'}=1$. Fix $\theta_n\in\mathbb D$.
By \eqref{Omega-representation}, the points of
$V_{S'}\cap\mathbf{\Theta}_n$ lying over $\theta_n$ are precisely the
joint eigenvalues of the commuting $(n-1)$-tuple
$
(\Phi_1(\theta_n),\ldots,\Phi_{n-1}(\theta_n)).$
Since these are $N\times N$ matrices, their joint point spectrum is
finite. Hence every fibre of $V_{S'}\cap\mathbf{\Theta}_n$ under the
projection onto the $\theta_n$-coordinate is finite. By
\cite[Lemma 3.5]{S. Pal 1}, it follows that $V_{S'}$ is an affine
algebraic curve. Thus
$
\dim V_{S'}=1.$
Since $V_{S'}$ is defined in $\mathbb C^n$ by $n-1$ polynomials and has
dimension one, the polynomials
\[
f_1,\ldots,f_{n-1}
\]
have height $n-1$. After localization at each maximal ideal containing
$(f_1,\ldots,f_{n-1})$, the ambient polynomial ring is a regular, hence
Cohen--Macaulay, local ring. Therefore
$f_1,\ldots,f_{n-1}$ form a regular sequence. Now $\Omega$ is a nonempty relatively open subset of the algebraic
curve $V_S$. Hence $\Omega$ is Zariski dense in $V_S$. Since
$
\Omega\subseteq V_{S'}\cap\mathbf{\Theta}_n,$
every polynomial in the ideal defining $V_{S'}$ vanishes on $\Omega$ and
therefore vanishes on $V_S$. Thus
$
V_S\subseteq V_{S'}.$
Since both $V_S$ and $V_{S'}$ have dimension one, every irreducible
component of $V_S$ is an irreducible component of $V_{S'}$.
In particular, the irreducible algebraic curve containing $\Omega$ is
a component of the complete intersection $V_{S'}$. Therefore
$\Omega$ is a part of an affine algebraic curve which is a
set-theoretic complete intersection.

It follows from the fact that
$
(T_{\Phi_1},\ldots,T_{\Phi_{n-1}},T_z)$
is a pure $\mathbf{\Theta}_n$-isometry on
$H^2(\mathcal D_{M_{\theta_n}^*})$ that
$T_{\Phi_1},\ldots,T_{\Phi_{n-1}}$ commute. Since the
$\Phi_i$'s are operator-valued analytic functions, we have
\[
T_{\Phi_i}T_{\Phi_j}=T_{\Phi_i\Phi_j},
\quad
T_{\Phi_j}T_{\Phi_i}=T_{\Phi_j\Phi_i}.\]
Consequently,
\[
\Phi_i(z)\Phi_j(z)=\Phi_j(z)\Phi_i(z),
\quad z\in\mathbb D.
\]
Since $\Phi_i$ and $\Phi_j$ are polynomials, the above identity
holds for every $z\in\mathbb C$.
Recall that
\[
\Phi_i(z)=\sum_{l=0}^{p}A_l^{(i)}z^l,
\quad 1\leq i\leq n-1.
\]
Therefore,
\[
0=[\Phi_i(z),\Phi_j(z)]
=
\sum_{k=0}^{2p}
\left(
\sum_{l=0}^{k}
[A_l^{(i)},A_{k-l}^{(j)}]
\right)z^k,
\]
where we put $A_l^{(i)}=0$ whenever
$l\notin\{0,\ldots,p\}$. Comparing coefficients of powers of $z$,
we obtain
\begin{equation}\label{commutator}
\sum_{l=0}^{k}[A_l^{(i)},A_{k-l}^{(j)}]=0,
\qquad
0\leq k\leq2p,\quad
1\leq i,j\leq n-1.
\end{equation}
Thus condition $(1)$ of Theorem~\ref{Determinantal Representation}
is satisfied.

Next, by \cite[Theorem 3.5]{Paul 1},
$
(T_{\Phi_1},\ldots,T_{\Phi_{n-1}},T_z)$
is a $\mathbf{\Theta}_n$-contraction for every $z\in\mathbb T$.
Hence, by Lemma~\ref{Lem 4},
$
(\Phi_1(z),\ldots,\Phi_{n-1}(z),zI)$
is a $\mathbf{\Theta}_n$-contraction for every $z\in\mathbb T$.
For $z\in\overline{\mathbb D}$, define
\[
\Psi(z)
=
(\Phi_1(z),\ldots,\Phi_{n-1}(z),zI)
\]
as an operator-valued $n$-tuple on
$\mathcal D_{M_{\theta_n}^*}$. Since each $\Phi_i$ is a polynomial,
$\Psi$ is holomorphic on a neighbourhood of
$\overline{\mathbb D}$.
Let $q$ be any polynomial in $n$ variables. For $z\in\mathbb T$,
we have
\[
\|q(\Psi(z))\|
=
\|q(\Phi_1(z),\ldots,\Phi_{n-1}(z),zI)\|
\leq
\|q\|_{\infty,\overline{\mathbf{\Theta}}_n}.
\]
The function
$
z\longmapsto q(\Psi(z))$
is an operator-valued analytic function on a neighbourhood of
$\overline{\mathbb D}$. Hence, by the maximum modulus principle,
\[
\|q(\Psi(z))\|
\leq
\|q\|_{\infty,\overline{\mathbf{\Theta}}_n},
\quad z\in\mathbb D.
\]
Thus
$
(\Phi_1(z),\ldots,\Phi_{n-1}(z),zI)$
is a $\mathbf{\Theta}_n$-contraction for every $z\in\mathbb D$. Consequently, for every $z\in\mathbb D$,
\[
\sigma_T\big(\Phi_1(z),\ldots,\Phi_{n-1}(z),zI\big)
\subseteq \overline{\mathbf{\Theta}}_n.
\]
We claim that, in fact,
\[
\sigma_T\big(\Phi_1(z),\ldots,\Phi_{n-1}(z),zI\big)
\subseteq \mathbf{\Theta}_n,
\quad z\in\mathbb D.
\]
Suppose, to the contrary, that for some $z_0\in\mathbb D$ there exists
\[
(\theta_1,\ldots,\theta_{n-1},z_0)
\in
\sigma_T\big(\Phi_1(z_0),\ldots,\Phi_{n-1}(z_0),z_0I\big)
\cap\partial\mathbf{\Theta}_n.
\]
Since
$
\theta_i\in\sigma_T(\Phi_i(z_0)),
1\le i\le n-1,$
we have
\[
\det(\Phi_i(z_0)-\theta_iI)=0,
\quad 1\le i\le n-1.
\]
Hence
$
(\theta_1,\ldots,\theta_{n-1},z_0)$
belongs to the affine algebraic variety determined by the
determinantal equations defining the algebraic curve containing
$\Omega$. Therefore, as $\Omega$ is an irreducible component of this
curve and is open in the relative topology of $\mathbf{\Theta}_n$,
the above point belongs to the closure of $\Omega$.

On the other hand, since $|z_0|<1$, \cite[Theorem 2.5]{Biswas 2}
implies that
$
(\theta_1,\ldots,\theta_{n-1},z_0)
\notin b\mathbf{\Theta}_n.$
Thus
$
(\theta_1,\ldots,\theta_{n-1},z_0)
\in
\overline{\Omega}\cap
\big(\partial\mathbf{\Theta}_n\setminus b\mathbf{\Theta}_n\big),$
which contradicts the fact that $\Omega$ is a distinguished variety
in $\mathbf{\Theta}_n$. Therefore,
$
\sigma_T\big(\Phi_1(z),\ldots,\Phi_{n-1}(z),zI\big)
\subseteq \mathbf{\Theta}_n, z\in\mathbb D.$
Now fix $\theta_n\in\mathbb D$ and let $v$ be a unit joint eigenvector
of
$
\big(\Phi_1(\theta_n),\ldots,\Phi_{n-1}(\theta_n)\big)$
corresponding to the joint eigenvalue
$(\theta_1,\ldots,\theta_{n-1})$. Then
\[
\Phi_i(\theta_n)v=\theta_i v,
\quad 1\le i\le n-1,
\]
and hence
\[
\big(\Phi_1(\theta_n),\ldots,\Phi_{n-1}(\theta_n),
\theta_nI\big)v
=
(\theta_1,\ldots,\theta_{n-1},\theta_n)v.
\]
Consequently,
$
(\theta_1,\ldots,\theta_{n-1},\theta_n)
\in\mathbf{\Theta}_n.$
By \cite[Theorem 3.1]{Paul}, there exists
$
(\alpha_1,\ldots,\alpha_{n-1})\in\mathbb G_{n-1}$
such that
\[
\theta_i
=
\alpha_i+\overline{\alpha}_{n-i}\theta_n^p,
\quad
1\le i\le n-1.
\]
Moreover,
\[
\theta_i
=
\langle\Phi_i(\theta_n)v,v\rangle
=
\left\langle
\sum_{l=0}^{p}A_l^{(i)}\theta_n^l v,v
\right\rangle.
\]
Therefore,
\[
\sum_{l=0}^{p}
\langle A_l^{(i)}v,v\rangle\theta_n^l
=
\alpha_i+\overline{\alpha}_{n-i}\theta_n^p,
\qquad
1\le i\le n-1.
\]
Thus condition $(2)$ of Theorem~\ref{Determinantal Representation}
is satisfied. This completes the proof.

\end{proof}

We now show that every distinguished variety in $\mathbf{\Theta}_n$
arises as the image, under the map $\theta$, of a distinguished variety
in $\mathbb{D}^n$, and that the converse also holds.

\begin{thm}\label{Thata_n D^n Dist Vari Relation}
Let $\Omega\subseteq\mathbf{\Theta}_n$. Then $\Omega$ is a distinguished
variety in $\mathbf{\Theta}_n$ if and only if there exists a distinguished
variety $\widehat{\Omega}$ in $\mathbb{D}^n$ such that
\[
\Omega=\boldsymbol{\theta}(\widehat{\Omega}).
\]
\end{thm}

\begin{proof}
Suppose first that $\Omega$ is a distinguished variety in
$\mathbf{\Theta}_n$. By Theorem~\ref{Determinantal Representation},
$\Omega$ is an affine algebraic curve and a complete intersection.
Let
\[
\Omega
=
V_S\cap\mathbf{\Theta}_n,
\quad
S=\{f_1,\ldots,f_n\}
\subseteq\mathbb C[\theta_1,\ldots,\theta_n].
\]
Define
\[
\widehat{\Omega}
=
\left\{
(z_1,\ldots,z_n)\in\mathbb D^n:
(f_i\circ\theta)(z_1,\ldots,z_n)=0,
\quad 1\leq i\leq n
\right\}.
\]
Since $\boldsymbol{\theta}(\mathbb D^n)=\mathbf{\Theta}_n$, it follows immediately that
$
\boldsymbol{\theta}(\widehat{\Omega})=\Omega.$
We claim that $\widehat{\Omega}$ is one-dimensional. Suppose, to the
contrary, that $\dim\widehat{\Omega}>1$. Then, by the same argument used
in the proof of Theorem~\ref{Determinantal Representation}, an algebraic
set of dimension greater than one cannot exit $\mathbb D^n$ only through
the distinguished boundary $\mathbb T^n$. Consequently,
$\widehat{\Omega}$ must intersect
$
\partial\mathbb D^n\setminus\mathbb T^n.$
Let
$
z=(z_1,\ldots,z_n)\in
\overline{\widehat{\Omega}}\cap
(\partial\mathbb D^n\setminus\mathbb T^n).$
Since $\boldsymbol{\theta}$ is continuous on $\overline{\mathbb D}^{\,n}$, we have
$
\boldsymbol{\theta}(z)\in
\overline{\Omega}\cap
\left(
\overline{\mathbf{\Theta}}_n\setminus b\mathbf{\Theta}_n
\right).$
Indeed, $\boldsymbol{\theta}(\mathbb T^n)=b\mathbf{\Theta}_n$, whereas
$z\notin\mathbb T^n$. This contradicts the fact that $\Omega$ is a
distinguished variety in $\mathbf{\Theta}_n$. Hence
$
\dim\widehat{\Omega}=1.$
Moreover,
$
\boldsymbol{\theta}(\mathbb T^n)=b\mathbf{\Theta}_n.$
Since $\Omega$ exits $\mathbf{\Theta}_n$ through
$b\mathbf{\Theta}_n$, there is a point
$
\lambda\in
\overline{\Omega}\cap b\mathbf{\Theta}_n.$
Choose $z\in\mathbb T^n$ such that $\boldsymbol{\theta}(z)=\lambda$. By the definition
of $\widehat{\Omega}$ and the continuity of $\boldsymbol{\theta}$, we obtain
$
z\in\overline{\widehat{\Omega}}\cap\mathbb T^n.$
Thus $\widehat{\Omega}$ exits $\mathbb D^n$ through $\mathbb T^n$.
Therefore $\widehat{\Omega}$ is a distinguished variety in $\mathbb D^n$.

Conversely, suppose that $\widehat{\Omega}$ is a distinguished variety
in $\mathbb D^n$. Since $\boldsymbol{\theta}$ is a polynomial map,
$
\Omega:=\theta(\widehat{\Omega})$
is an algebraic set in $\mathbf{\Theta}_n$. Since
$\widehat{\Omega}$ is one-dimensional, $\Omega$ is also one-dimensional.
Furthermore,
$
\boldsymbol{\theta}(\mathbb T^n)=b\mathbf{\Theta}_n.$
Hence, as $\widehat{\Omega}$ exits $\mathbb D^n$ through $\mathbb T^n$,
its image $\Omega$ exits $\mathbf{\Theta}_n$ through
$b\mathbf{\Theta}_n$. Thus $\Omega$ is a distinguished variety in
$\mathbf{\Theta}_n$.
Therefore,
\[
\Omega\text{ is a distinguished variety in }\mathbf{\Theta}_n
\quad\Longleftrightarrow\quad
\Omega=\boldsymbol{\theta}(\widehat{\Omega})
\]
for some distinguished variety $\widehat{\Omega}$ in $\mathbb D^n$.
\end{proof}

The following theorem establishes the polynomial convexity of the closure of every distinguished variety in $\mathbf{\Theta}_n$.
\begin{thm}\label{Polynomial Convexity}
	The closure of every distinguished variety in $\mathbf{\Theta}_n$ is polynomially convex.
\end{thm}
\begin{proof}
By Theorem~\ref{Determinantal Representation}, every distinguished
variety $\Omega$ in $\mathbf{\Theta}_n$ admits a representation
\[
\begin{aligned}
\Omega
=
\Big\{&
(\theta_1,\ldots,\theta_n)\in\mathbf{\Theta}_n:
(\theta_1,\ldots,\theta_{n-1})
\in
\sigma_T(\Phi_1(\theta_n),\ldots,\Phi_{n-1}(\theta_n))
\Big\},
\end{aligned}
\]
where
\[
\Phi_i(z)=\sum_{l=0}^{p}A_l^{(i)}z^l,
\quad 1\leq i\leq n-1.
\]
Put
\[
f_i(\theta_1,\ldots,\theta_n)
=
\det(\Phi_i(\theta_n)-\theta_iI),
\quad 1\leq i\leq n-1.
\]
Then $f_i$ vanishes identically on $\Omega$ and hence on
$\overline{\Omega}$. In particular,
$
\|f_i\|_{\infty,\overline{\Omega}}=0.$
Since
$
\Omega\subseteq\mathbf{\Theta}_n
\subseteq\overline{\mathbf{\Theta}}_n,$
we have
$
\overline{\Omega}\subseteq\overline{\mathbf{\Theta}}_n.$
Let
$
\widetilde{x}\in\mathbb C^n\setminus\overline{\Omega}.$
We distinguish two cases.

\medskip
\noindent\textbf{Case 1.}
Suppose
$
\widetilde{x}\notin\overline{\mathbf{\Theta}}_n.$
Since $\overline{\mathbf{\Theta}}_n$ is polynomially convex, there
exists a polynomial $\widetilde p$ such that
\[
|\widetilde p(\widetilde{x})|
>
\|\widetilde p\|_{\infty,\overline{\mathbf{\Theta}}_n}
\geq
\|\widetilde p\|_{\infty,\overline{\Omega}}.
\]
Thus $\widetilde{x}$ does not belong to the polynomial hull of
$\overline{\Omega}$.

\medskip
\noindent\textbf{Case 2.}
Suppose
$
\widetilde{x}\in
\overline{\mathbf{\Theta}}_n\setminus\overline{\Omega}.$
By the determinantal representation and the equality
\[
Z(f_1,\ldots,f_{n-1})\cap\overline{\mathbf{\Theta}}_n
=
\overline{\Omega},
\]
there exists $k\in\{1,\ldots,n-1\}$ such that
$
f_k(\widetilde{x})\neq0.$
Since $f_k$ vanishes on $\overline{\Omega}$,
$
\|f_k\|_{\infty,\overline{\Omega}}=0,$
whereas
$
|f_k(\widetilde{x})|>0.$
Hence $f_k$ separates $\widetilde{x}$ from $\overline{\Omega}$.
In both cases every point of
$\mathbb C^n\setminus\overline{\Omega}$ can be separated from
$\overline{\Omega}$ by a polynomial. Therefore
$\overline{\Omega}$ is polynomially convex.
\end{proof}

\section{Dilation and Functional Model for a Class of Pure
$\mathbf{\Theta}_n$-Contractions}\label{Dilation and Model}
Every $\mathbf{\Theta}_n$-isometry is the restriction of a
$\mathbf{\Theta}_n$-unitary to a joint invariant subspace, it follows
immediately from Definition~\ref{Dilation Definiton} that every
$\mathbf{\Theta}_n$-contraction which admits a
$\mathbf{\Theta}_n$-isometric dilation also admits a
$\mathbf{\Theta}_n$-unitary dilation.

\begin{defn}\label{Dilation Definiton}
	Let $\mathbf{T}=(T_1,\dots,T_n)$ be a $\mathbf{\Theta}_n$-contraction
	on a Hilbert space $\mathcal{H}$. An $n$-tuple of commuting bounded
	operators $\mathbf{V}=(V_1,\dots,V_n)$ acting on a Hilbert space
	$\mathcal{K}$ containing $\mathcal{H}$ as a subspace is called a
	\emph{$\mathbf{\Theta}_n$-isometric dilation} of $\mathbf{T}$ if
	\begin{enumerate}
		\item $\mathbf{V}$ is a $\mathbf{\Theta}_n$-isometry;
		\item $V_i^*|_{\mathcal H}=T_i^*$, for $1\leq i\leq n$.
	\end{enumerate}
	Furthermore, $\mathbf{V}$ is called a
	\emph{minimal $\mathbf{\Theta}_n$-isometric dilation} of $\mathbf{T}$ if
	\[
	\mathcal K
	=
	\overline{\operatorname{span}}
	\{V_n^k h:\ h\in\mathcal H,\ k\in\mathbb N\cup\{0\}\}.
	\]
	In this case, $\mathcal K$ is called the
	\emph{minimal $\mathbf{\Theta}_n$-isometric dilation space} of
	$\mathbf{T}$.
\end{defn}
In this section, we develop a minimal $\mathbf{\Theta}_n$-isometric
dilation and a functional model for a class of pure
$\mathbf{\Theta}_n$-contractions satisfying the condition
\begin{equation}\label{Condition}
	D_{T_n^*}T_i^*
	=
	\sum_{l=0}^{p}
	A_l^{(i)*}D_{T_n^*}T_n^{*l},
	\quad 1\leq i\leq n-1.
\end{equation}
\begin{prop}\label{Existence of Minimal Isometric Dilation}
	Let $\mathbf{T}=(T_1,\dots,T_n)$ be a
	$\mathbf{\Theta}_n$-contraction on a Hilbert space $\mathcal H$.
	Suppose that $\mathbf{T}$ admits a $\mathbf{\Theta}_n$-isometric
	dilation. Then $\mathbf{T}$ admits a minimal
	$\mathbf{\Theta}_n$-isometric dilation.
\end{prop}

\begin{proof}
	Suppose that
	$\mathbf{V}=(V_1,\dots,V_n)$ is a
	$\mathbf{\Theta}_n$-isometric dilation of $\mathbf{T}$ acting on a
	Hilbert space $\mathcal K\supseteq\mathcal H$. Define
	\[
	\mathcal K_0
	=
	\overline{\operatorname{span}}
	\{V_n^k h:\ h\in\mathcal H,\ k\in\mathbb N\cup\{0\}\}.
	\]
	By the structural relations satisfied by a
	$\mathbf{\Theta}_n$-isometry, $\mathcal K_0$ is invariant under
	$V_1,\dots,V_n$. Hence, putting
	\[
	V_{0i}=V_i|_{\mathcal K_0},\quad 1\leq i\leq n,
	\]
	we obtain a commuting $n$-tuple
	$\mathbf V_0=(V_{01},\dots,V_{0n})$ on $\mathcal K_0$.
Since $\mathcal K_0$ is invariant under $\mathbf V$, the tuple
	$\mathbf V_0$ is a $\mathbf{\Theta}_n$-contraction. Moreover,
	$V_{0n}$ is an isometry, being the restriction of the isometry
	$V_n$ to the invariant subspace $\mathcal K_0$. Therefore, by
	\cite[Theorem 3.5]{Paul 1},
	$\mathbf V_0$ is a $\mathbf{\Theta}_n$-isometry.
For $h\in\mathcal H$ and $1\leq i\leq n$, we have
	\[
	V_{0i}^*h
	=
	V_i^*h
	=
	T_i^*h.
	\]
	Thus $\mathbf V_0$ is a $\mathbf{\Theta}_n$-isometric dilation of
	$\mathbf T$. Finally, by construction,
	\[
	\mathcal K_0
	=
	\overline{\operatorname{span}}
	\{V_{0n}^kh:\ h\in\mathcal H,\ k\in\mathbb N\cup\{0\}\}.
	\]
	Hence $\mathbf V_0$ is a minimal
	$\mathbf{\Theta}_n$-isometric dilation of $\mathbf T$.
\end{proof}

\begin{prop}\label{Minimal Co-isometric Dilation}
	Let $\mathbf{T}=(T_1,\dots,T_n)$ be a
	$\mathbf{\Theta}_n$-contraction on a Hilbert space $\mathcal H$.
	Let $\mathbf{V}=(V_1,\dots,V_n)$ be an $n$-tuple of operators on a
	Hilbert space $\mathcal K\supseteq\mathcal H$. Then $\mathbf V$ is a
	minimal $\mathbf{\Theta}_n$-isometric dilation of $\mathbf T$ if and
	only if
	$
	\mathbf V^*=(V_1^*,\dots,V_n^*)$
	is a minimal $\mathbf{\Theta}_n$-co-isometric dilation of
	$
	\mathbf T^*=(T_1^*,\dots,T_n^*).$
\end{prop}

\begin{proof}
	Suppose first that $\mathbf V$ is a minimal
	$\mathbf{\Theta}_n$-isometric dilation of $\mathbf T$. By definition,
	$\mathbf V$ is a $\mathbf{\Theta}_n$-isometry and
	\[
	V_i^*|_{\mathcal H}=T_i^*,\quad 1\leq i\leq n.
	\]
	Thus $
	(V_1^*,\dots,V_n^*)$
	is a $\mathbf{\Theta}_n$-co-isometry and is a dilation of
	$\mathbf T^*$. Moreover, by the minimality of $\mathbf V$,
	\[
	\mathcal K
	=
	\overline{\operatorname{span}}
	\{V_n^k h:\ h\in\mathcal H,\ k\geq0\}.
	\]
	This is precisely the corresponding minimality condition for the
	adjoint tuple, with the distinguished operator replaced by
	$V_n^*$.

	Conversely, suppose that $\mathbf V^*$ is a
	$\mathbf{\Theta}_n$-co-isometric dilation of $\mathbf T^*$. Then,
	by taking adjoints, $\mathbf V$ is a
	$\mathbf{\Theta}_n$-isometry and
	\[
	V_i^*|_{\mathcal H}=T_i^*,\qquad 1\leq i\leq n.
	\]
	Hence $\mathbf V$ is a $\mathbf{\Theta}_n$-isometric dilation of
	$\mathbf T$. The minimality condition for $\mathbf V^*$ is
	equivalent, by taking adjoints, to the minimality condition for
	$\mathbf V$. Therefore $\mathbf V$ is a minimal
	$\mathbf{\Theta}_n$-isometric dilation of $\mathbf T$.
\end{proof}
Let $\mathbf{T}=(T_1,\dots,T_n)$ be a pure
$\mathbf{\Theta}_n$-contraction. Define
$
W:\mathcal H\longrightarrow
H^2(\mathbb D)\otimes\mathcal D_{T_n^*}$
by
\begin{equation}\label{W}
	W h
	=
	\sum_{k\geq0}z^k\otimes D_{T_n^*}T_n^{*k}h,
	\quad h\in\mathcal H.
\end{equation}
Since $T_n$ is a pure contraction, the standard defect-space
calculation shows that $W$ is an isometry. Indeed,
\[
\begin{aligned}
\|Wh\|^2
&=
\sum_{k\geq0}
\|D_{T_n^*}T_n^{*k}h\|^2\\
&=
\sum_{k\geq0}
\left(
\|T_n^{*k}h\|^2-\|T_n^{*(k+1)}h\|^2
\right)\\
&=
\|h\|^2,
\end{aligned}
\]
where we have used the purity of $T_n^*$.
We first determine $W^*$. For
$\eta\in\mathcal D_{T_n^*}$ and $k\geq0$, we have
\[
\begin{aligned}
\langle W^*(z^k\otimes\eta),h\rangle
&=
\left\langle
z^k\otimes\eta,
\sum_{m\geq0}z^m\otimes
D_{T_n^*}T_n^{*m}h
\right\rangle\\
&=
\langle\eta,D_{T_n^*}T_n^{*k}h\rangle\\
&=
\langle T_n^kD_{T_n^*}\eta,h\rangle.
\end{aligned}
\]
Consequently,
\begin{equation}\label{W*}
	W^*(z^k\otimes\eta)
	=
	T_n^kD_{T_n^*}\eta,
	\quad k\geq0.
\end{equation}
\begin{thm}\label{Dilation}
	Let $\mathbf{T}=(T_1,\dots,T_n)$ be a pure
	$\mathbf{\Theta}_n$-contraction on a Hilbert space $\mathcal H$.
	Suppose that
	\[
	\left\{
	A_l^{(i)}:
	0\leq l\leq p,\ 1\leq i\leq n-1
	\right\}
	\subseteq\mathcal B(\mathcal D_{T_n^*})
	\]
	satisfies 
	\begin{equation}\label{Condition}
	A_l^{(i)}
	=
	A_{p-l}^{(n-i)*} \quad \text{and}\quad D_{T_n^*}T_i^*
	=
	\sum_{l=0}^{p}
	A_l^{(i)*}D_{T_n^*}T_n^{*l},
	\quad 1\leq i\leq n-1.
	\end{equation}
	Consider the operators on
	$
	\mathcal K
	=
	H^2(\mathbb D)\otimes\mathcal D_{T_n^*}$
	given by
	\begin{equation}\label{V_i,V_n}
	V_i
	=
	\sum_{l=0}^{p}M_z^l\otimes A_l^{(i)},
	\quad 1\leq i\leq n-1,\quad \text{and}\quad
	V_n=M_z\otimes I_{\mathcal D_{T_n^*}}.
\end{equation}
	If $W$ is defined by \eqref{W}, then $W$ is an isometry and
	\[
	W^*V_i=T_iW^*,
	\quad 1\leq i\leq n.
	\]
	Moreover, if
	$
	\left(
	\gamma_1\sum_{l=0}^{p}A_l^{(1)}z^l,\dots,
	\gamma_{n-1}\sum_{l=0}^{p}A_l^{(n-1)}z^l
	\right)$
	is a $\Gamma_{n-1}$-contraction for every $z\in\mathbb T$, then
	$\mathbf V$ is a minimal pure $\mathbf{\Theta}_n$-isometric
	dilation of $\mathbf T$.
\end{thm}
\begin{proof}
By the Sz.-Nagy dilation theorem \cite{Nagy}, $V_n=M_z\otimes
I_{\mathcal D_{T_n^*}}$ is the minimal pure isometric dilation of
$T_n$. We first prove that
\[
W^*V_i=T_iW^*,\qquad 1\leq i\leq n.
\]
For $1\leq i\leq n-1$, $k\geq0$, and
$\eta\in\mathcal D_{T_n^*}$, we have
\begin{equation}\label{D 2}
\begin{aligned}
W^*V_i(z^k\otimes\eta)
&=
W^*\left(
\sum_{l=0}^{p}M_z^l\otimes A_l^{(i)}
\right)(z^k\otimes\eta)\\
&=
W^*\left(
\sum_{l=0}^{p}z^{k+l}\otimes A_l^{(i)}\eta
\right)\\
&=
\sum_{l=0}^{p}
T_n^{k+l}D_{T_n^*}A_l^{(i)}\eta\\
&=
T_n^k
\left(
\sum_{l=0}^{p}
T_n^lD_{T_n^*}A_l^{(i)}\eta
\right).
\end{aligned}
\end{equation}
Taking adjoints in \eqref{Condition}, we obtain
\[
T_iD_{T_n^*}
=
\sum_{l=0}^{p}
T_n^lD_{T_n^*}A_l^{(i)}.
\]
Therefore, from \eqref{D 2},
\begin{equation}\label{D 3}
\begin{aligned}
W^*V_i(z^k\otimes\eta)
&=
T_n^k
\left(
\sum_{l=0}^{p}
T_n^lD_{T_n^*}A_l^{(i)}\eta
\right)\\
&=
T_n^kT_iD_{T_n^*}\eta\\
&=
T_iT_n^kD_{T_n^*}\eta\\
&=
T_iW^*(z^k\otimes\eta).
\end{aligned}
\end{equation}
Here we have used the commutativity of the tuple $\mathbf T$.
For $i=n$, we have
\begin{equation}\label{D 4}
\begin{aligned}
W^*V_n(z^k\otimes\eta)
&=
W^*(z^{k+1}\otimes\eta)\\
&=
T_n^{k+1}D_{T_n^*}\eta\\
&=
T_n\left(T_n^kD_{T_n^*}\eta\right)\\
&=
T_nW^*(z^k\otimes\eta).
\end{aligned}
\end{equation}
Since the linear span of the vectors
$
\{z^k\otimes\eta:\ \eta\in\mathcal D_{T_n^*},\ k\geq0\}$
is dense in
$H^2(\mathbb D)\otimes\mathcal D_{T_n^*}$, it follows from
\eqref{D 3} and \eqref{D 4} that
\[
W^*V_i=T_iW^*,\quad 1\leq i\leq n.
\]
It remains to show that $\mathbf V$ is a minimal
$\mathbf{\Theta}_n$-isometric dilation. First, by the relation
$
A_l^{(i)}=A_{p-l}^{(n-i)*},$
we have
\[
V_i=V_{n-i}^*V_n^p,\quad 1\leq i\leq n-1.
\]
Moreover, suppose that
\[
\left(
\gamma_1\sum_{l=0}^{p}A_l^{(1)}z^l,\dots,
\gamma_{n-1}\sum_{l=0}^{p}A_l^{(n-1)}z^l
\right)
\]
is a $\Gamma_{n-1}$-contraction for every $z\in\mathbb T$. Then the
operator-valued polynomials
\[
\Phi_i(z):=\sum_{l=0}^{p}A_l^{(i)}z^l,
\quad 1\leq i\leq n-1,
\]
commute with one another for every $z\in\mathbb T$. Since each
$\Phi_i$ is analytic, the corresponding multiplication operators
also commute. Under the identification
$
H^2(\mathcal D_{T_n^*})
=
H^2(\mathbb D)\otimes\mathcal D_{T_n^*},$
we consequently have that
\[
\left(
\gamma_1\sum_{l=0}^{p}M_z^l\otimes A_l^{(1)},\dots,
\gamma_{n-1}\sum_{l=0}^{p}M_z^l\otimes A_l^{(n-1)}
\right)
\]
is a $\Gamma_{n-1}$-contraction for every $z\in\mathbb T$.
Therefore, by \cite[Theorem 3.10]{Biswas 2},
$
\mathbf V=(V_1,\dots,V_n)$
is a $\mathbf{\Theta}_n$-isometry on
$H^2(\mathbb D)\otimes\mathcal D_{T_n^*}$. Since
$
V_n=M_z\otimes I_{\mathcal D_{T_n^*}}$
is a pure isometry, $\mathbf V$ is a pure
$\mathbf{\Theta}_n$-isometry.
Finally,
\[
H^2(\mathbb D)\otimes\mathcal D_{T_n^*}
=
\overline{\operatorname{span}}
\{V_n^k(1\otimes\eta):
\eta\in\mathcal D_{T_n^*},\ k\geq0\}.
\]
Since
$
\overline{\operatorname{Ran}}D_{T_n^*}
=
\mathcal D_{T_n^*},$
and
$
W^*(1\otimes\eta)=D_{T_n^*}\eta,$
the dilation space is generated by the powers of $V_n$ applied to
the embedded copy of $\mathcal H$. Hence $\mathbf V$ is minimal.
Therefore, $\mathbf V$ is a minimal pure
$\mathbf{\Theta}_n$-isometric dilation of $\mathbf T$.
\end{proof}
We shall use the following standard consequence of the
Sz.-Nagy-Foias model theory; see
\cite[Lemma 4.5]{S. Pal 3} for a proof.

\begin{lem}\label{W isometry}
	Let $T_n$ be a contraction. Then
	\begin{equation}\label{W Property}
		WW^*+
		M_{\Theta_{T_n}}M_{\Theta_{T_n}}^*
		=
		I_{H^2(\mathbb D)\otimes\mathcal D_{T_n^*}}.
	\end{equation}
\end{lem}

We now obtain a functional model for a special class of pure
$\mathbf{\Theta}_n$-contractions.

\begin{thm}\label{Functional Model}
	Let $\mathbf{T}=(T_1,\dots,T_n)$ be a pure
	$\mathbf{\Theta}_n$-contraction on a Hilbert space $\mathcal H$.
	Suppose that
	\[
	\left\{A_l^{(i)}:
	0\leq l\leq p,\ 1\leq i\leq n-1,\ 
	A_l^{(i)}=A_{p-l}^{(n-i)*}\right\}
	\]
	is a collection of operators on $\mathcal D_{T_n^*}$ satisfying
	\eqref{Condition}. Then
	\begin{enumerate}
		\item for $1\leq i\leq n-1$,
		$
		T_i~\text{is unitarily equivalent to}~
		P_{\mathcal H_{T_n}}
		\left(\sum_{l=0}^{p}M_z^l\otimes A_l^{(i)}\right)
		\big|_{\mathcal H_{T_n}};$
		
		\item $
		T_n~\text{is unitarily equivalent to}~
		P_{\mathcal H_{T_n}}
		(M_z\otimes I_{\mathcal D_{T_n^*}})
		\big|_{\mathcal H_{T_n}},$
	\end{enumerate}
	where
	$
	\mathcal H_{T_n}
	=
	\big(H^2(\mathbb D)\otimes\mathcal D_{T_n^*}\big)
	\ominus
	M_{\Theta_{T_n}}
	\big(H^2(\mathbb D)\otimes\mathcal D_{T_n}\big).$
\end{thm}

\begin{proof}
	Since $W$ is an isometry, $WW^*$ is the orthogonal projection
	onto $\operatorname{Ran}W$. On the other hand, since $T_n$ is
	pure, its characteristic function $\Theta_{T_n}$ is inner and
	hence $M_{\Theta_{T_n}}$ is an isometry. Therefore
	$M_{\Theta_{T_n}}M_{\Theta_{T_n}}^*$ is the orthogonal projection
	onto
	$
	M_{\Theta_{T_n}}
	\big(H^2(\mathbb D)\otimes\mathcal D_{T_n}\big).$
	It follows from Lemma~\ref{W isometry} that
	\[
	\operatorname{Ran}W
	=
	\left(
	H^2(\mathbb D)\otimes\mathcal D_{T_n^*}
	\right)
	\ominus
	M_{\Theta_{T_n}}
	\left(
	H^2(\mathbb D)\otimes\mathcal D_{T_n}
	\right)
	=
	\mathcal H_{T_n}.
	\]
	Thus $W$ is a unitary operator from $\mathcal H$ onto
	$\mathcal H_{T_n}$.
By Theorem~\ref{Dilation},
	\[
	W^*V_i=T_iW^*,\quad 1\leq i\leq n,
	\]
	where
	\[
	V_i=\sum_{l=0}^{p}M_z^l\otimes A_l^{(i)},
	\quad 1\leq i\leq n-1,\quad \text{and}\quad
	V_n=M_z\otimes I_{\mathcal D_{T_n^*}}.
	\]
	Multiplying the intertwining relations on the right by $W$, we
	obtain
	\[
	W^*V_iW=T_i,\quad 1\leq i\leq n.
	\]
	Since $W$ is a unitary operator from $\mathcal H$ onto
	$\mathcal H_{T_n}$, this shows that
	\[
	T_i
	\cong
	P_{\mathcal H_{T_n}}
	\left(\sum_{l=0}^{p}M_z^l\otimes A_l^{(i)}\right)
	\big|_{\mathcal H_{T_n}},
	\quad 1\leq i\leq n-1,\quad \text{and}\quad
	T_n
	\cong
	P_{\mathcal H_{T_n}}
	(M_z\otimes I_{\mathcal D_{T_n^*}})
	\big|_{\mathcal H_{T_n}}.
	\]
	This completes the proof.
\end{proof}
\section{Matricial von Neumann Inequality on the Distinguished Varieties}
\label{von Neumann Inequality}

In this section, we establish a matricial von Neumann type inequality
for a class of $\mathbf{\Theta}_n$-contractions on distinguished
varieties in $\mathbf{\Theta}_n$.

\begin{thm}\label{von Neumann inequality}
	Let $\mathbf{T}=(T_1,\dots,T_n)$ be a
	$\mathbf{\Theta}_n$-contraction on a Hilbert space $\mathcal H$
	such that $T_n^*$ is pure. Suppose that
	\[
	\left\{A_l^{(i)}:
	0\leq l\leq p,\ 1\leq i\leq n-1,\ 
	A_l^{(i)}=A_{p-l}^{(n-i)*}\right\}
	\]
	is a collection of operators on $\mathcal D_{T_n^*}$ satisfying
	\begin{equation}\label{Condition *}
	\sum_{l=0}^{k}
	[A_l^{(i)},A_{k-l}^{(j)}]=0,
	\quad D_{T_n^*}T_i^*
	=
	\sum_{l=0}^{p}
	A_l^{(i)*}D_{T_n^*}T_n^{*\,l},
	\quad 0\leq k\leq 2p,\quad
	1\leq i,j\leq n-1.
	\end{equation}
	and $$
	\left(
	\gamma_1\displaystyle\sum_{l=0}^{p}A_l^{(1)}z^l,\dots,
	\gamma_{n-1}\displaystyle\sum_{l=0}^{p}A_l^{(n-1)}z^l
	\right)$$
	is a $\Gamma_{n-1}$-contraction for every $z\in\mathbb T$.
	Suppose further that the matrices $A_l^{(i)}$ determine a
	distinguished variety $\Omega_{\mathbf T}$ through the
	determinantal representation of Theorem~\ref{Determinantal Representation}.
	Then, for every matrix-valued polynomial
	$
	P\in\mathcal M_q\big(\mathbb C[z_1,\dots,z_n]\big),
 q\geq1,$
	we have
	\begin{equation}\label{von Neumann 1}
		\|P(T_1,\dots,T_n)\|
		\leq
		\sup_{\boldsymbol{\theta}\in
		\overline{\Omega}_{\mathbf T}\cap b\mathbf{\Theta}_n}
		\|P(\boldsymbol{\theta})\|,
	\end{equation}
	and
	\begin{equation}\label{von Neumann 2}
		\|P(T_1^*,\dots,T_n^*)\|
		\leq
		\sup_{\boldsymbol{\theta}\in
		\overline{\Omega}_{\mathbf T}\cap b\mathbf{\Theta}_n}
		\|P(\boldsymbol{\theta})\|.
	\end{equation}
\end{thm}
\begin{proof}
	Since $\dim\mathcal D_{T_n^*}<\infty$, put $
		m=\dim\mathcal D_{T_n^*}.$
	Thus each $A_l^{(i)}$ is an $m\times m$ matrix. Since the
	operators $A_l^{(i)}$ determine the distinguished variety
	$\Omega_{\mathbf T}$, Theorem~\ref{Determinantal Representation}
	gives
	\[
	f_i(\theta_1,\ldots,\theta_n)
	=
	\det(\Phi_i(\theta_n)-\theta_iI),
	\quad 1\leq i\leq n-1,
	\]
	where $
	\Phi_i(z)=\sum_{l=0}^p A_l^{(i)}z^l,$
	and
	$
	\Omega_{\mathbf T}
	=
	V(f_1,\ldots,f_{n-1})\cap\mathbf\Theta_n.$
	We first observe that the commutator condition
	\[
	\sum_{l=0}^{k}[A_l^{(i)},A_{k-l}^{(j)}]=0,
	\quad 0\leq k\leq2p,
	\]
	implies
	\[
	[\Phi_i(z),\Phi_j(z)]=0,
	\quad z\in\mathbb C.
	\]
	 Since $
	\left(
	\gamma_1\displaystyle\sum_{l=0}^{p}A_l^{(1)}z^l,\dots,
	\gamma_{n-1}\displaystyle\sum_{l=0}^{p}A_l^{(n-1)}z^l
	\right)$ is a $\Gamma_{n-1}$-contraction for every $z\in\mathbb T$, it implies that $(\gamma_1\Phi_1(z),\ldots,\gamma_{n-1}\Phi_{n-1}(z))$
is a $\Gamma_{n-1}$-contraction for every $z\in\mathbb T$. 
Moreover, the relation
$
A_l^{(i)}=A_{p-l}^{(n-i)*}$
implies, for $z\in\mathbb T$,
\[
\Phi_i(z)=z^p\Phi_{n-i}(z)^*,
\quad 1\leq i\leq n-1.
\]
Consequently,
\[
M_{\Phi_i}
=
M_{\Phi_{n-i}}^*M_z^p,
\quad 1\leq i\leq n-1,
\]
on $L^2(\mathcal D_{T_n})$.
Since $M_z$ is unitary on $L^2(\mathcal D_{T_n})$, it follows from
\cite[Theorem 3.2]{Biswas 2} that
$
(M_{\Phi_1},\ldots,M_{\Phi_{n-1}},M_z)$
is a $\mathbf{\Theta}_n$-unitary on
$L^2(\mathcal D_{T_n})$.
Equivalently, for every $z\in\mathbb T$,
$
(\Phi_1(z),\ldots,\Phi_{n-1}(z),zI)$
is a commuting tuple of normal operators whose Taylor joint spectrum
is contained in $b\mathbf{\Theta}_n$.

On the other hand, since the operators $A_l^{(i)}$ determine the
distinguished variety $\Omega_{\mathbf T}$, we have
$
V_S\cap b\mathbf{\Theta}_n
=
\partial\Omega_{\mathbf T},$
where
\[
S=\{f_i:1\leq i\leq n-1\},
\quad
f_i(\theta_1,\ldots,\theta_n)
=
\det(\Phi_i(\theta_n)-\theta_iI).
\]
Therefore, for every $z\in\mathbb T$,
\[
\sigma_T(\Phi_1(z),\ldots,\Phi_{n-1}(z),zI)
\subseteq
V_S\cap b\mathbf{\Theta}_n
=
\partial\Omega_{\mathbf T}.
\]
Thus
$
(\Phi_1(z),\ldots,\Phi_{n-1}(z),zI)$
is a $\mathbf{\Theta}_n$-unitary for every $z\in\mathbb T$.
By \cite[Lemma 5.8]{S. Pal 1}, $\partial\Omega_{\mathbf T}$ is a
spectral set for
$
(M_{\Phi_1},\ldots,M_{\Phi_{n-1}},M_z).$
Since
$
(\gamma_1\Phi_1(z),\ldots,
\gamma_{n-1}\Phi_{n-1}(z))$
is a $\Gamma_{n-1}$-contraction for every $z\in\mathbb T$, Theorem
\ref{Dilation} yields that
$
(M_{\Phi_1},\ldots,M_{\Phi_{n-1}},M_z)$
is a $\mathbf{\Theta}_n$-isometric dilation of
$
\mathbf T^*=(T_1^*,\ldots,T_n^*).$
Consequently,
$
(M_{\Phi_1}^*,\ldots,M_{\Phi_{n-1}}^*,M_z^*)$
is a $\mathbf{\Theta}_n$-co-isometric dilation of
$\mathbf T$. This yields
\begin{equation*}
\begin{aligned}
T_{\Phi_i}^*|_{\mathcal H}&=T_i,
\quad 1\leq i\leq n-1,\quad
T_z^*|_{\mathcal H}&=T_n.
\end{aligned}
\end{equation*}
Since
$
(M_{\Phi_1},\ldots,M_{\Phi_{n-1}},M_z)$
is a $\mathbf{\Theta}_n$-unitary on
$L^2(\mathcal D_{T_n})$, its restriction to the joint invariant
subspace $H^2(\mathcal D_{T_n})$ is a
$\mathbf{\Theta}_n$-isometry. Hence
$
(T_{\Phi_1},\ldots,T_{\Phi_{n-1}},T_z)$
is a $\mathbf{\Theta}_n$-isometric dilation of
$
\mathbf T^*=(T_1^*,\ldots,T_n^*).$
Consequently,
$
(T_{\Phi_1}^*,\ldots,T_{\Phi_{n-1}}^*,T_z^*)$
is a $\mathbf{\Theta}_n$-co-isometric dilation of $\mathbf T$. Let $P$ be a matrix-valued polynomial in $n$ variables and define
$\widehat P$ by
\[
\widehat P(X_1,\ldots,X_n)
=
P(X_1^*,\ldots,X_n^*)^*.
\]
Then
\[
\begin{aligned}
\|P(T_1,\ldots,T_n)\|
&=
\|P(T_{\Phi_1}^*|_{\mathcal H},
       \ldots,
       T_{\Phi_{n-1}}^*|_{\mathcal H},
       T_z^*|_{\mathcal H})\|\\
&\leq
\|P(T_{\Phi_1}^*,\ldots,T_{\Phi_{n-1}}^*,T_z^*)\|\\
&=
\|\widehat P(T_{\Phi_1},\ldots,T_{\Phi_{n-1}},T_z)\|.
\end{aligned}
\]
Since
$
(T_{\Phi_1},\ldots,T_{\Phi_{n-1}},T_z)$
is the restriction to $H^2(\mathcal D_{T_n})$ of the multiplication
tuple
$
(M_{\Phi_1},\ldots,M_{\Phi_{n-1}},M_z)$
on $L^2(\mathcal D_{T_n})$, we obtain
\[
\begin{aligned}
\|P(T_1,\ldots,T_n)\|
&\leq
\|\widehat P(M_{\Phi_1},\ldots,M_{\Phi_{n-1}},M_z)\|\\
&=
\sup_{|\lambda|=1}
\|\widehat P(
\Phi_1(\lambda),\ldots,\Phi_{n-1}(\lambda),\lambda I
)\|.
\end{aligned}
\]
We consider the variety
\begin{equation*}
\begin{aligned}
\Omega_{\mathbf T}^*
=
\Big\{&
(\theta_1,\ldots,\theta_{n-1},\theta_n)\in\mathbf{\Theta}_n:
(\theta_1,\ldots,\theta_{n-1})
\in
\sigma_T(\widetilde{\Phi}_1(\theta_n),
\ldots,\widetilde{\Phi}_{n-1}(\theta_n))
\Big\},
\end{aligned}
\end{equation*}
where
$
\widetilde{\Phi}_i(z)
=
\sum_{l=0}^p A_l^{(i)*}z^l, 1\leq i\leq n-1.$
By the coefficient symmetry
$
A_l^{(i)}=A_{p-l}^{(n-i)*},$
we have, for $|z|=1$,
\[
\widetilde{\Phi}_i(z)
=
z^p\Phi_{n-i}(\overline z).
\]
In particular, the boundary variety associated with the adjoint
coefficients is obtained from the original boundary variety by the
corresponding conjugation of the spectral parameters. Hence
\[
\overline{\Omega}_{\mathbf T}^*\cap b\mathbf{\Theta}_n
=
\left\{
(\overline{\theta}_1,\ldots,\overline{\theta}_n):
(\theta_1,\ldots,\theta_n)
\in
\overline{\Omega}_{\mathbf T}\cap b\mathbf{\Theta}_n
\right\}.
\tag{*}
\]
Now let $p$ be a matrix-valued polynomial and define
\[
\widehat p(X_1,\ldots,X_n)
=
p(X_1^*,\ldots,X_n^*)^*.
\]
Using the normality of
$
(M_{\Phi_1},\ldots,M_{\Phi_{n-1}},M_z)$
on $L^2(\mathcal D_{T_n})$, together with the spectral mapping theorem,
we obtain
\begin{equation*}
\begin{aligned}
\|p(T_1,\ldots,T_n)\|
&\leq
\sup_{\lambda\in\mathbb T}
\sup_{\boldsymbol\theta\in
\sigma_T(\Phi_1(\lambda),\ldots,\Phi_{n-1}(\lambda),
\lambda I)}
\|\widehat p(\boldsymbol\theta)\|\\
&=
\sup_{\boldsymbol\theta\in
\overline{\Omega}_{\mathbf T}^*\cap b\mathbf{\Theta}_n}
\|\widehat p(\boldsymbol\theta)\|.
\end{aligned}
\end{equation*}
For every $\boldsymbol\theta\in\mathbb C^n$,
$
\widehat p(\boldsymbol\theta)
=
p(\overline{\boldsymbol\theta})^*,$
and hence
$
\|\widehat p(\boldsymbol\theta)\|
=
\|p(\overline{\boldsymbol\theta})\|.$
Consequently, by $(*)$,
\begin{equation*}
\begin{aligned}
\|p(T_1,\ldots,T_n)\|
&\leq
\sup_{\boldsymbol\theta\in
\overline{\Omega}_{\mathbf T}^*\cap b\mathbf{\Theta}_n}
\|p(\overline{\boldsymbol\theta})\|\\
&=
\sup_{\boldsymbol\theta\in
\overline{\Omega}_{\mathbf T}\cap b\mathbf{\Theta}_n}
\|p(\boldsymbol\theta)\|.
\end{aligned}
\end{equation*}
This proves \eqref{von Neumann 1}. This completes the proof.
	\end{proof}

	\begin{rem}
It follows from the proof of Theorem \ref{von Neumann inequality} that both
$\mathbf{T}$ and $\mathbf{T}^*$ admit normal boundary dilations whose
joint spectra are contained in the distinguished boundary of
$\mathbf{\Theta}_n$. Moreover, the corresponding boundary spectra are
contained in the distinguished varieties
$\overline{\Omega}_{\mathbf T}\cap b\mathbf{\Theta}_n$ and
$\overline{\Omega}_{\mathbf T}^*\cap b\mathbf{\Theta}_n$, respectively.
\end{rem}

	\noindent (A. Gupta) \sc{Department of Mathematics, IIT Bhilai, 6th Lane Road, Jevra, Chhattisgarh 491002}\\
	{E-mail address:} {aparnagupta@iitbhilai.ac.in}
	
	\vspace{.5cm}
	
	\noindent (S. Mandal) \sc{Department of Mathematics, IIT Bhilai, 6th Lane Road, Jevra, Chhattisgarh 491002}\\
	{E-mail address:} {shubhankarm@iitbhilai.ac.in}
	
	\vspace{.5cm}
	
	\noindent (A. Pal) \sc{Department of Mathematics, IIT Bhilai, 6th Lane Road, Jevra, Chhattisgarh 491002}\\
	{E-mail address:} {avijit@iitbhilai.ac.in}
	
	\vspace{.5cm}
	
	\noindent (B. Paul) \sc{Department of Mathematics, IIT Bhilai, 6th Lane Road, Jevra, Chhattisgarh 491002}\\
	{E-mail address:} {bhaskarpaul@iitbhilai.ac.in}  
\end{document}